\documentclass[12pt]{amsart}
\usepackage{amsmath,amssymb,latexsym, amsthm, amscd, mathrsfs, stmaryrd, mathtools, yhmath}
\usepackage[linktocpage=true]{hyperref}
\usepackage{cleveref}
\usepackage{xcolor}
\usepackage[all]{xy}
\usepackage{enumitem}
\usepackage{chngcntr}
\usepackage{multirow}
\usepackage{array}
\usepackage{booktabs}
\usepackage{tikz}
\usepackage{tikz-cd}
\usepackage{bbm}
\crefformat{section}{\S#2#1#3} 
\crefformat{subsection}{\S#2#1#3}
\crefformat{subsubsection}{\S#2#1#3}

\newtheorem{thm}{Theorem} [section]
\theoremstyle{definition}

\newtheorem{Def}[thm]{Definition}
\newtheorem{example}[thm]{Example}
\newtheorem{rem}[thm]{Remark}
\theoremstyle{plain}
\newtheorem{prop}[thm]{Proposition}

\newtheorem{lem}[thm]{Lemma}
\newtheorem{cor}[thm]{Corollary}
\newtheorem{conj}[thm]{Conjecture}

\numberwithin{equation}{section}

\newcommand{\Hom}{\mathrm{Hom}}

\newcommand{\mrm}{\mathrm}

\newcommand{\C}{\mathbb C}

\newcommand{\End}{\mrm{End}}

\newcommand{\g}{\mathfrak{g}}

\newcommand{\KK}{\mathbb K}

\newcommand{\Z}{\mathbb Z}

\makeatletter
\newcommand{\extp}{\@ifnextchar^\@extp{\@extp^{\,}}}
\def\@extp^#1{\mathop{\bigwedge\nolimits^{\!#1}}}
\makeatother

\DeclareMathOperator*{\Rep}{Rep}
\DeclareMathOperator*{\Vecc}{Vec}
\DeclareMathOperator*{\sVec}{sVec}
\DeclareMathOperator{\Ver}{Ver}
\DeclareMathOperator*{\tr}{tr}

\title[Lectures on Symmetric Tensor Categories]{Lectures on Symmetric Tensor Categories}

\author[Pavel Etingof]{Pavel Etingof}
\address{Department of Mathematics, Massachusetts Institute of Technology, Cambridge, MA 02139} \email{etingof@math.mit.edu}

\author[Arun S. Kannan]{Arun S. Kannan}
\address{Department of Mathematics, Massachusetts Institute of Technology, Cambridge, MA 02139} \email{akannan@mit.edu}

\begin{document}

\begin{abstract} This is an expanded version of the notes by the second author of the lectures on symmetric tensor categories given by the first author at Ohio State University in March 2019 and later at ICRA-2020 in November 2020. We review some aspects of the current state of the theory of symmetric  tensor categories and discuss their applications. 
\end{abstract}

\maketitle

\setcounter{tocdepth}{1}
\tableofcontents

\section{Introduction}

A modern view of representation theory is that it is a study not just of individual representations (say, finite dimensional representations of an affine group or, more generally, supergroup scheme $G$ over an algebraically closed field $\KK$) but also of the {\bf category} ${\rm Rep}(G)$ they form. The properties of ${\rm Rep}(G)$ can be summarized by saying that it is a {\bf symmetric tensor category} (shortly, STC) which uniquely determines $G$. A STC is a natural home for studying any kind of linear-algebraic structures (commutative algebras, Lie algebras, Hopf algebras, modules over them, etc.); for instance, doing so in ${\rm Rep}(G)$ amounts to studying such structures with a $G$-symmetry. It is therefore natural to ask: does the study of STC reduce to (super)group representation theory, or is it more general? In other words, do there exist STCs other than ${\rm Rep}(G)$ (or, more precisely, its slight generalization $\Rep(G,z)$)\footnote{Here $z\in G(\KK)$ is such that $z^2=1$ and $z$ acts on the function algebra $\mathcal O(G)$ by parity. The category $\Rep(G,z)$ consists of representations of $G$ on supervector spaces on which $z$ acts by parity.}? If so, this would be interesting, since algebra in such STCs would be a new kind of algebra, one ``without vector spaces".

Luckily, the answer turns out to be ``yes". In Lecture 1, after reviewing the general theory of STC, we will discuss such examples in characteristic zero (so-called {\bf Deligne categories}), and also {\bf Deligne's theorem} (\cite{D1}), which puts restrictions on the kind of examples one can have.

Examples of symmetric tensor categories over complex numbers which are not representation categories of supergroups were first given by P. Deligne and J. S. Milne in 1981 (\cite{DM}, Examples 1.26, 1.27). These very interesting categories are interpolations of representation categories of classical groups $GL_n$, $O_n$, $Sp_n$ to arbitrary complex values of $n$. P. Deligne later generalized them to symmetric groups $S_n$ (\cite{D2}) and also to characteristic $p$ (\cite{D3}, see also \cite{H}, 3.3), where, somewhat unexpectedly, one needs to interpolate $n$ to $p$-adic integer values rather than elements of the ground field. These categories are now known under the umbrella name ``Deligne categories". In Lecture 2 we will review
the structure of Deligne categories $\Rep GL_t$ interpolating the classical representation categories $\Rep GL_n(\Bbb C)$ and discuss their alternative construction using ultrafilters.
This construction allows one to generalize these categories to characteristic $p$.

Deligne categories discussed in Lecture 2 violate an obvious necessary condition for a STC to have any realization by finite dimensional vector spaces (and in particular to be of the form ${\rm Rep}(G,z)$): for each object $X$ the length of the $n$-th tensor power $X^{\otimes n}$ of $X$ grows at most exponentially with $n$. We call this property {\bf moderate growth}. So it is natural to ask if there exist STCs of moderate growth other than ${\rm Rep}(G,z)$. In characteristic zero, the negative answer is given by Deligne's theorem, discussed in Lecture 1. Namely Deligne's theorem says that a STC of moderate growth can always be realized in {\bf supervector spaces.} However, in characteristic $p$ the situation is much more interesting. Namely, Deligne's theorem is known to fail in any characteristic $p>0$. The simplest exotic symmetric tensor category of moderate growth (i.e., not of the form ${\rm Rep}(G)$) for $p>3$ is the semisimplification of the category of representations of $\Bbb Z/p$, called the {\bf Verlinde category} ${\rm Ver}_p$, which was first considered in \cite{GK,GM}. For example, for $p=5$, this category has an object $X$ such that $X^{\otimes 2}=\mathbbm{1}\oplus X$, so $X$ cannot be realized by a vector space (as its dimension would have to equal $\frac{1\pm \sqrt{5}}{2}$). In Lecture 3 we will discuss the notion of semisimplification, the Verlinde category, Ostrik's generalization of Deligne's theorem for fusion categories in characteristic $p$ (\cite{O}) and recent further extension to semisimple (and, more generally, Frobenius exact) symmetric tensor categories (\cite{CEO}). We will also discuss the Verlinde categories for prime powers, ${\rm Ver}_{p^n}$, constructed in \cite{BE, BEO} and a conjectural generalization of Deligne's theorem to general symmetric tensor categories of moderate growth in characteristic $p$ which involves $\Ver_{p^n}$. In the Appendix we discuss some applications of these techniques to modular representation theory which have not been considered previously, and study dimensions in ribbon categories in positive characteristic.  

{\bf Acknowledgements.} P.E. is grateful to K. Coulembier, Dave Benson, Victor Ostrik for useful discussions. P. E.'s work was partially supported by the NSF grant DMS - 1916120. These lecture notes are based upon work supported by The National Science Foundation Graduate Research Fellowship Program under Grant No. 1842490 awarded to the second author.

\section{Lecture 1: Symmetric tensor categories and Deligne's theorem}

The original goal of representation theory (going back to the works of Frobenius, Schur, Weyl, Brauer and others) was constructing and understanding representations of various groups and Lie algebras. In modern representation theory, however, the perspective has shifted to look at collections of representations and the morphisms between them simultaneously, rather than focus on individual representations. Hence, the fundamental object of study is the {\bf category of representations} of a group. In order to generalize, the most important structures and properties of this category are identified to define something known as a {\bf symmetric tensor category}. In the next subsection, we will build up to the definition of a symmetric tensor category by isolating these key structures and properties, using the category of representations of a group as the prototypical example. For more details on tensor categories, in particular symmetric ones, we refer the reader to \cite{EGNO}.

\subsection{Symmetric tensor categories}
Let $G$ be a group and $\KK$ be an algebraically closed field of any characteristic. Let $\Rep_\KK (G)$ be the category of finite dimensional representations of $G$ over $\KK$. The morphisms between two $G$-modules $V,W$ are the usual $G$-module homomorphisms: $\Hom_{\Rep_\KK(G)}(V, W) = \{\phi:V \rightarrow W \ : \phi \circ g = g \circ \phi \ \ \forall g \in G \}$. It is natural to ask: {\bf what structures and properties does this category have?}

\subsubsection{Additive structure}
Some things we notice immediately are:

\begin{enumerate}
  \item[$\bullet$]  $\Rep_\KK (G)$ is {\bf $\KK$-linear}, meaning that the $\Hom$-sets are  $\KK$-vector spaces and the composition of morphisms is bilinear.
  \item[$\bullet$]  $\Rep_\KK (G)$ is an {\bf abelian category}. While the technical definition is rather long, the gist is that this category has finite direct sums, kernels, cokernels, and images of morphisms. One interpretation comes via the Freyd-Mitchell theorem, which states states that any abelian category can be realized (albeit non-uniquely) as a full subcategory of the category of modules over some ring $R$; in our case, $R$ can be taken to be the group algebra of $G$ over $\KK$.
  \item[$\bullet$]  $\Rep_\KK (G)$ is {\bf artinian} (or {\bf locally finite}), which means that objects have finite length (i.e. a finite Jordan-Holder or composition series) and $\Hom$ spaces are finite dimensional over $\KK$.
\end{enumerate}
These encapsulate the additive structure of the category $\Rep_\KK(G)$.

Often one needs to consider additive categories which are not artinian or even abelian, but only {\bf Karoubian}, which means that they have images of idempotents and finite direct sums. An object $X\ne 0$ of such a category is {\bf indecomposable} if it is not a direct sum of two nonzero objects. The {\bf Krull-Schmidt theorem} guarantees that in a Karoubian category with finite dimensional Hom spaces, any object has a unique (up to an isomorphism) decomposition into a direct sum of indecomposables.

\begin{rem} An object $X$ of an abelian category $\mathcal C$ is {\bf simple}
if it has no subobjects other than $0$ and $X$.  The category $\mathcal C$ is called {\bf semisimple} if every indecomposable object is simple (i.e., every object is semisimple, that is, isomorphic to a direct sum of simple objects). If such ${\mathcal C}$ is linear over an algebraically closed field $\KK$ and has finite dimensional Hom spaces then for indecomposable objects $X,Y$, $\Hom(X,Y)=0$ if $X\ncong Y$, and $\Hom(X,X)=\KK$, a property often called ``Schur's lemma".

In ${\rm Rep}_\KK(G)$, simple objects are irreducible representations and indecomposable objects are indecomposable representations. So if $G$ is a finite group and ${\rm char}(\KK)$ does not divide $|G|$ then $\Rep_\KK(G)$ is semisimple. In this case,
``Schur's lemma" is the classical Schur lemma in representation theory.
\end{rem}

\subsubsection{Monoidal structure}
The category ${\rm Rep}_\KK(G)$ also has a {\bf monoidal structure}, which encodes the tensor product of representations.

\begin{Def}[Monoidal Structure]
Let $\mathcal{C}$ be a category. A monoidal structure on $\mathcal{C}$ is a triple $(\otimes, a, \mathbbm{1})$, where $\otimes$ is a bifunctor $\otimes: \mathcal{C} \times \mathcal{C} \rightarrow \mathcal{C}$ called {\bf the tensor product}, $a$ is a natural isomorphism of functors $a: (- \otimes -) \otimes - \rightarrow - \otimes (- \otimes -)$ from $\mathcal{C} \times \mathcal{C} \times \mathcal{C}$ to $\mathcal{C}$ called the {\bf associativity constraint} or {\bf associativity isomorphism}, and $\mathbbm{1}$ is an object of $\mathcal{C}$ called the {\bf unit object}, which satisfy the following axioms.

\begin{enumerate}
  \item \textbf{The pentagon axiom}. For all $W, X, Y, Z \in \mathcal{C}$, the following diagram commutes:

  \[
    \begin{tikzpicture}[commutative diagrams/every diagram]
      \node(P0)at (90 + 0*72:4.1cm){$((W\otimes X)\otimes Y)\otimes Z$};
      \node(P1)at (90 + 1*72:4cm){$(W\otimes (X\otimes Y))\otimes Z$};
      \node(P2)at (90 + 2*72:3cm){\makebox[5ex][r]{$W\otimes ((X\otimes Y)\otimes Z)$}};
      \node(P3)at (90 + 3*72:3cm){\makebox[5ex][l]{$W\otimes (X\otimes (Y\otimes Z))$}};
      \node(P4)at (90 + 4*72:4cm){$(W\otimes X)\otimes(Y\otimes Z)$};

      \path[commutative diagrams/.cd,every arrow,every label](P0)edgenode[swap] {  $a_{W,X,Y} \otimes 1_Z$} (P1)(P1)edgenode[swap] {  $a_{W, X\otimes Y, Z}$} (P2)(P2)edgenode{$1_W \otimes a_{X, Y, Z}$} (P3)(P4)edgenode{$a_{W,X,Y\otimes Z}$} (P3)(P0)edgenode{$a_{W\otimes X, Y, Z}$} (P4);
    \end{tikzpicture}
  \]

  \item \textbf{The unit axiom}. There is an isomorphism $\iota: \mathbbm 1\otimes \mathbbm 1\to \mathbbm 1$, and the functors \linebreak $R_{\mathbbm{1}}: X \mapsto X\otimes \mathbbm{1}$ and $L_{\mathbbm{1}}: X \mapsto \mathbbm{1} \otimes X$ are autoequivalences of $\mathcal{C}$.
\end{enumerate}
\end{Def}

\begin{rem} Given $(\otimes,a)$ satisfying the associativity axiom, the object $\mathbbm{1}$ and isomorphism $\iota$, if they exist, are uniquely determined up to a unique isomorphism, so the unit object is a property and not a structure.
\end{rem}

A category endowed with a monoidal structure is called a {\bf monoidal category}. This categorifies the notion of a monoid, in the sense that the set of isomorphism classes of an (essentially small) monoidal category is a monoid under $\otimes$ with unit $\mathbbm{1}$; namely, the existence of $a$ implies the associativity of multiplication (i.e., that we can ignore parentheses in a product). On the other hand, the pentagon axiom for $\otimes$ is a higher coherence property that we don't see for usual monoids. It insures that we can ignore parentheses even at the level of objects themselves, not just their isomorphism classes: any two ways to pass from one parenthesization of a tensor product $X_1\otimes...\otimes X_n$ to another give the same isomorphism (this is called the {\bf Mac Lane coherence theorem}).

In our example of $\Rep_\KK (G)$, the $\otimes$ is the usual tensor product of representations with the usual associativity isomorphism, and $\mathbbm{1}$ is given by the one-dimensional trivial representation of $G$.

\subsubsection{Symmetric structure}\label{brai}
We also observe that $\Rep_\KK(G)$ has a certain symmetry structure. Namely, for any two objects $V, W \in \Rep_\KK(G)$, we have an isomorphism $V \otimes W \cong W \otimes V$ given by the usual swap, and it commutes with morphisms of representations. Hence, categorically speaking, we have a functorial isomorphism $c_{V,W} : V \otimes W \rightarrow W \otimes V$. This motivates the following definitions.

\begin{Def}[Braided Monoidal Category] A {\bf braided monoidal category} is a monoidal category $\mathcal{C}$ endowed with a natural isomorphism of functors $c : (- \otimes -) \rightarrow  (- \otimes^{\rm op} - )$ from $\mathcal{C} \times \mathcal{C}$ to $\mathcal{C}$ called a {\bf braiding} (where op denotes the opposite tensor product, i.e., $X\otimes^{\rm op}Y=Y\otimes X$) such that the following hexagon diagrams commute for any objects $X, Y, Z \in \mathcal{C}$:

\[
  \begin{tikzcd}[column sep={2cm,between origins}, row sep={1.732050808cm,between origins}]
      & X \otimes (Y \otimes Z) \arrow[rr, "c_{X, Y \otimes Z}"]  && (Y \otimes Z) \otimes X \arrow[rd, "a_{Y,Z,X}"] &  \\
      (X \otimes Y) \otimes Z \arrow[ru, "a_{X, Y, Z}"', swap] \arrow[rd, "c_{X, Y} \otimes 1_Z"']&  &&  & Y\otimes (Z \otimes X) \\
      & (Y \otimes X) \otimes Z \arrow[rr, "a_{Y,X,Z}"'] && Y \otimes (X \otimes Z) \arrow[ru, "1_Y \otimes c_{X,Z}"'] &
  \end{tikzcd}
\]

\[
\begin{tikzcd}[column sep={2cm,between origins}, row sep={1.732050808cm,between origins}]
    & (X \otimes Y) \otimes Z \arrow[rr, "c_{X\otimes Y, Z}"]  && Z \otimes (X \otimes Y) \arrow[rd, "a_{Z,X,Y}^{-1}"] &  \\
    X \otimes (Y \otimes Z) \arrow[ru, "a_{X, Y, Z}^{-1}"', swap] \arrow[rd, "1_X \otimes c_{Y, Z}"']&  &&  & (Z\otimes X) \otimes Y \\
    & X \otimes (Z \otimes Y) \arrow[rr, "a_{X,Z,Y}^{-1}"'] && (X \otimes Z) \otimes Y \arrow[ru, "c_{X,Z} \otimes 1_Y"'] &
\end{tikzcd}
\]

\end{Def}

In a braided monoidal category, for any object $X$ we have a natural action of the braid group $B_n$ on $X^{\otimes n}$ given by $b_i\mapsto c_{i,i+1}$, where $b_i$ are the standard generators of $B_n$. This explains the terminology.

The notion of a braided monoidal category categorifies the notion of a commutative monoid, formed by isomorphism classes of objects of such a category; for this reason the braiding is also known as the {\bf commutativity constraint}.

It is clear that the category ${\rm Rep}_\KK(G)$ is a braided monoidal category, with braiding given by the swap $V\otimes W\to W\otimes V$. However, in the categorical world (unlike the usual world of sets) there is an even stronger version of commutativity, still enjoyed by this category.

\begin{Def}[Symmetric Monoidal Category] A braided monoidal category with braiding $c$ is a {\bf symmetric monoidal category} if for all objects $X, Y \in \mathcal{C}$, \[c_{Y, X} \circ c_{X, Y} = 1_{X \otimes Y}.\]
In this case, we call the braiding {\bf symmetric}.
\end{Def}
If $X$ is an object in a symmetric monoidal category, the above action of the braid group $B_n$ on $X^{\otimes n}$ factors through the symmetric group $S_n$, as we have $b_i^2=1$. This explains the terminology.

We see that $\Rep_{\KK}(G)$ is a symmetric monoidal category.

\subsubsection{Rigid categories}

The category $\Rep_\KK(G)$ also has {\bf dual objects}, namely if $V$ is a $G$-module, then the dual space $V^*$ is a $G$-module with action $(gf)(v) = f(g^{-1}v)$ for all $g \in G, v \in V, f \in V^*$.
This means that the symmetric monoidal category ${\rm Rep}_\KK(G)$ satisfies the property called {\bf rigidity}.

\begin{Def}[Rigidity]
Let $(\mathcal{C}, \otimes, a, \mathbbm{1}, c)$ be a symmetric monoidal category, and let $X \in \mathcal{C}$ be an object. An object $X^* \in \mathcal{C}$ is called the {\bf dual} of $X$ if there exist morphisms $ev_X: X^* \otimes X \rightarrow \mathbbm{1}$ and $coev_X: \mathbbm{1} \rightarrow X \otimes X^*$, called the {\bf evaluation and coevaluation morphisms}, such that the following two compositions are the identity morphisms:

\begin{align*}
  X \xrightarrow{coev_X \otimes 1_X} (X\otimes X^*) \otimes X \xrightarrow{a_{X, X^*, X}} X \otimes (X^* \otimes X) \xrightarrow{1_X \otimes ev_X} X \\
  X^* \xrightarrow{1_{X^*} \otimes coev_X} X^* \otimes (X\otimes X^*) \xrightarrow{a^{-1}_{X^*, X, X^*}} (X^* \otimes X) \otimes X^* \xrightarrow{ev_X \otimes 1_{X^*}} X^*
\end{align*}
where the unit isomorphisms like that between $X$ and $X \otimes \mathbbm{1}$ are suppressed. An object that admits a dual is called {\bf rigid}. The category $\mathcal{C}$ is called rigid if all its objects are rigid.
\end{Def}

\begin{rem} The dual object $X^*$ equipped with the evaluation and coevaluation morphisms, if exists, is unique up to a unique isomorphism, which shows that rigidity of $X$ (and thus of $\mathcal C$) is a property and not a structure. Thus we have a contravariant functor $X\mapsto X^*$. Also, we have a functorial isomorphism $X\cong X^{**}$ induced by the symmetric braiding $c$ (i.e., $X$ has a natural structure of a dual object to $X^*$).
\end{rem}

\begin{rem} The notions of a rigid object and a rigid category can be defined for arbitrary (not necessarily braided) monoidal categories, in which case one should distinguish between {\bf left dual} $X^*$ and {\bf right dual} ${}^*X$, which are not always isomorphic, and consequently one does not have $X\cong X^{**}$  in general. However, since these lectures are about symmetric categories, we will not discuss this here and refer the reader to \cite{EGNO}, Subsection 2.10.
\end{rem}

\subsubsection{Symmetric tensor categories}
We note two final properties of $\Rep_{\KK}(G)$, namely

$\bullet$ the {\bf distributivity property}: $\otimes$ of morphisms is {\bf bilinear}, and

$\bullet$ $\End_{\Rep_\KK(G)}(\mathbbm{1}) = \KK$ (i.e. $\mathbbm{1}$ is indecomposable or, equivalently, simple).

The properties we have thus far observed are the basis for the definition of a {\bf symmetric tensor category}:

\begin{Def}[Symmetric Tensor Category]
A $\KK$-linear artinian\footnote{Recall that a $\KK$-linear category is artinian if it is abelian with objects of finite length and finite dimensional Hom spaces.} rigid symmetric monoidal category $(\mathcal{C}, \otimes, a, \mathbbm{1}, c)$ such that $\otimes$ is bilinear on morphisms\footnote{In such a category the tensor product is biexact, see \cite{EGNO}, Subsection 4.2.} and $\End_\mathcal{C} (\mathbbm{1}) = \KK$ is called a {\bf symmetric tensor category} over $\KK$.\footnote{The requirement that $\End_\mathcal{C} (\mathbbm{1}) = \KK$ is not very restrictive. Without this assumption, $\mathcal{C}$ will just be a direct sum of finitely many symmetric tensor categories over $\KK$.}
\end{Def}

We thus have

\begin{prop} ${\rm Rep}_\KK(G)$ is a symmetric tensor category over $\KK$.
\end{prop}

In particular, the simplest example of a symmetric tensor category is obtained when $G=\lbrace 1\rbrace$. In this case ${\rm Rep}_\KK(G)={\rm Vec}_\KK$, the {\bf category of finite dimensional vector spaces} over $\KK$.

\subsubsection{Monoidal functors and equivalences}

\par Monoids form a category where morphisms are homomorphisms of monoids, i.e.,
maps that preserve the multiplication and the unit. Similarly, we can define {\bf monoidal functors}, which are functors between monoidal categories that preserve the monoidal structure and the unit object. More precisely, we have

\begin{Def}[Monoidal Functor]
1.  If $(\mathcal{C}, \otimes, a, \mathbbm{1})$ and $(\mathcal{C}', \otimes', a', \mathbbm{1}')$ are two monoidal categories, a monoidal functor from $\mathcal{C}$ to $\mathcal{C}'$ is a pair $(F, J)$, where $F: \mathcal{C} \rightarrow \mathcal{C}'$ is a functor and $J: F(-)\otimes' F(-) \rightarrow F(-\otimes -)$ is a functorial isomorphism called the {\bf tensor structure} such that $F(\mathbbm{1}) \cong \mathbbm{1}'$ and the following hexagon diagram commutes for any objects $X, Y, Z \in \mathcal{C}$:

  \[
  \begin{tikzcd}[column sep={8cm,between origins}, row sep={1.5cm,between origins}]
      (F(X) \otimes' F(Y)) \otimes' F(Z) \arrow{r}{a_{F(X),F(Y),F(Z)}'} \arrow{d}{J_{X,Y}\otimes'1_{F(Z)}} & F(X) \otimes' (F(Y) \otimes' F(Z)) \arrow{d}{1_{F(X)}\otimes'J_{Y,Z}}  \\
      F(X\otimes Y) \otimes' F(Z) \arrow{d}{J_{X\otimes Y, Z}} & F(X) \otimes' F(Y\otimes Z) \arrow{d}{J_{X, Y\otimes Z}} \\
      F((X\otimes Y) \otimes Z) \arrow{r}{F(a_{X,Y,Z})} & F(X\otimes (Y\otimes Z)) \\
  \end{tikzcd}
  \]

2. If $\mathcal{C}$, $\mathcal{C}'$ are braided with braidings $c,c'$ then a monoidal functor $F: \mathcal C\to \mathcal C'$ is called {\bf braided} if it preserves the braiding, i.e., if for any $X,Y\in \mathcal{C}$ one has
$$
F(c_{X,Y})\circ J_{X,Y}=J_{Y,X}\circ c_{F(X),F(Y)}'.
$$
If moreover the braidings $c,c'$ are symmetric, a braided monoidal functor $\mathcal C\to \mathcal C'$ is called {\bf symmetric}.

3. If $\mathcal C,\mathcal C'$ are symmetric tensor categories then a {\bf symmetric tensor functor}
$F: \mathcal C\to \mathcal C'$ is an exact symmetric monoidal functor.\footnote{Such a functor is automatically faithful, see \cite{EGNO}, Remark 4.3.10.}
\end{Def}

\begin{rem} A symmetric monoidal functor automatically preserves duals, i.e., we have a canonical
isomorphism $F(X)^*\cong F(X^*)$.
\end{rem}

A symmetric tensor functor which is an equivalence of categories will be called a {\bf symmetric tensor equivalence}. One can now formulate one of our main problems.

\vskip .05in

{\bf Problem.} Classify symmetric tensor categories up to symmetric tensor equivalence, perhaps under some additional assumptions.

\subsubsection{Symmetric pseudotensor categories}

Often one has to consider categories with the same structures and properties
as symmetric tensor categories, except they are only Karoubian with finite dimensional Hom spaces rather than artinian. Such categories are called {\bf symmetric pseudotensor categories}. They are a lot more common than symmetric tensor categories since any additive category can be made Karoubian by the process of {\bf Karoubian completion}. Examples of symmetric pseudotensor categories which are not tensor
include categories of tilting modules for reductive groups in characteristic $p$, Deligne categories $\Rep GL_t$ for integer $t$, and many others. Some of these examples will be discussed below.

\subsection{Representations of Affine Group Schemes and Proalgebraic Completions}
Given that the definition of a symmetric tensor category is cooked up from the prototypical example $\Rep_\KK(G)$, we may ask ourselves if all symmetric tensor categories are of this form. The answer is ``no"; for instance, we can generalize and take $G$ to be an {\bf affine algebraic group} or, even more generally, an {\bf affine group scheme} over $\KK$.
\par
More concretely, this means that we have a commutative Hopf algebra $H = \mathcal{O}(G)$, which is the algebra of regular functions on $G$; then $G$ is recovered by $G = \mathrm{Spec} H$. We can interpret this as a functor $G(-)$ from commutative $\KK$-algebras to groups, where $G(R) = \Hom_{\KK-\mathrm{Alg}}(H, R)$, a group whose elements are points of $G$ over $R$. Finally, representations of $G$ are just finite-dimensional comodules of $H$.
\par
An affine algebraic group arises when the associated Hopf algebra is finitely generated and contains no nonzero nilpotent elements; the latter condition is guaranteed when $\mathrm{char}(\KK) = 0$ (\cite{EGNO}, Corollary 5.10.5).

Let us explain why this generalizes the previous setting of an abstract group. If we have a group $G$, we can define the (reduced) affine group scheme $\widehat{G}_\KK$ over $\KK$ called the \textbf{pro-algebraic completion} of $G$ such that $\Rep (\widehat{G}_\KK) \cong \Rep_\KK(G)$. Explicitly, $\mathcal{O}(\widehat{G}_\KK)$ is spanned by {\bf matrix coefficients} of finite-dimensional representations of $G$. Namely, if $V$ is a finite-dimensional representation of $G$, then a matrix coefficient is a function on $G$ of the form $\psi(g)  = \langle f, gv \rangle$, where $v \in V, f \in V^*$. Then the algebra spanned by all matrix coefficients is a Hopf algebra and coincides with the algebra of regular functions on $\widehat{G}_\KK$.
\par
Notice that this also includes the case of the category $\Rep(\g)$ of finite dimensional representations  of a {\bf Lie algebra} $\mathfrak{g}$ over $\KK$, as we can construct an affine group scheme in a similar way, such that its representation category is equivalent to that of $\g$. Namely, the corresponding Hopf algebra
$H$ is the subalgebra of $U(\g)^*$ spanned by matrix coefficients $a\mapsto \langle f,av\rangle$ where $a\in U(\g)$.

For more detail on such completions, see \cite{EGNO}, Example 5.4.4.

\subsection{The Category of Supervector Spaces}\label{supervec} It turns out that the representation categories of affine group schemes still do not exhaust all symmetric tensor categories; the simplest counterexample is the category of finite dimensional supervector spaces.
\par
Suppose that the characteristic of $\KK$ is not $2$. A supervector space is a $\Z/2$-graded vector space $V = V_0 \oplus V_1$. Elements of $V_0$ are called {\bf even}
and elements of $V_1$ are called {\bf odd}.
We can then define the category $\sVec_\KK$ of {\bf finite dimensional supervector spaces} over $\KK$ as the symmetric tensor category $\Rep_\KK (\Z/2)$ except that the braiding $c_{X,Y}: X\otimes Y \rightarrow Y \otimes X$ is given not by the usual swap but rather by the formula
$$
c_{X,Y}(x\otimes y) = (-1)^{|x||y|}(y \otimes x),
$$
where $x$ and $y$ are homogeneous (i.e. purely even or purely odd) and $|\cdot|$ denotes the parity of homogeneous elements. This is known as the {\bf Koszul sign rule}. Examples where supervector spaces arise are ubiquitous and include cohomology, differential forms, Lie superalgebras, etc.
\par

\subsection{Trace and dimension}
To show that the category of supervector spaces is really new, we will use the notion of
the dimension of a rigid object $X$ of a symmetric monoidal category.
This is a special case of the notion of the trace of an endomorphism
of $X$, which will also be important later.
\begin{Def}
  Let $X$ be a rigid object in a symmetric monoidal category $\mathcal{C}$, and suppose that $b: X\to X$ is a morphism.
  Define the {\bf trace} of $b$, denoted ${\rm Tr}(b)$,
  as the following composition:
  \[\mathbbm{1} \xrightarrow{coev_X} X \otimes X^* \xrightarrow{b\otimes 1_{X^*}}X\otimes X^*\xrightarrow{c_{X, X^*}} X^* \otimes X \xrightarrow{ev_X} \mathbbm{1}.\]
  Define the {\bf dimension} of $X$ by the formula ${\rm dim}(X)={\rm Tr}(1_X)$.
\end{Def}
  If $\End_\mathcal{C}(\mathbbm{1}) = \mathbb{K}$ (as happens, e.g.,  in symmetric pseudotensor categories over $\KK$) then the trace and hence the dimension are elements of $\KK$, and they generalize the notions of the trace of a linear operator and the dimension of a vector space. Indeed, in the case $\mathcal C=\Vecc_\KK$, if we fix a basis $\{x_i\}$ of $X$ and denote the dual basis by $\{x_i^*\}$, then the maps compose to give
  \[1 \mapsto \sum_{i=1}^{\dim_\KK X} x_i \otimes x_i^*\mapsto\sum_{i=1}^{\dim_\KK X} bx_i \otimes x_i^* \mapsto \sum_{i=1}^{\dim_\KK X} x_i^* \otimes bx_i   \mapsto \sum_{i=1}^{\dim_\KK X} (x_i^*, bx_i) = {\rm Tr}_X(b),\]
the usual trace of the linear endomorphism $b$. On the other hand, if $X = X_0 \oplus X_1$ is an object in $\sVec_\KK$, a similar argument shows that ${\rm Tr}(b)$ is the {\bf supertrace}:
$$
{\rm Tr}(b)={\rm tr}_{X_0}(b)-{\rm tr}_{X_1}(b).
$$
In particular, we have
$$
\dim X = \dim_\KK X_0 - \dim_\KK X_1
$$
where the LHS is the dimension in $\sVec_\KK$ and the RHS takes the usual vector space dimensions in $\Vecc_\KK$ and is then projected from $\Bbb Z$ to $\KK$.
\footnote{Note that in homological algebra this is precisely the Euler characteristic.}
In particular, if $X_0=0$ and $X_1$ is 1-dimensional then we have $X\otimes X\cong \mathbbm 1$
but $\dim(X)=-1$. On the other hand, in $\Rep(G)$ for an affine group scheme $G$, if $X\otimes X\cong \mathbbm{1}$ then $X$ is a 1-dimensional representation, so $\dim X=1$.
This shows that $\sVec_\KK$ is a genuinely new example (as ${\rm char}(\KK)\ne 2$, so $1\ne -1$ in $\KK$).

  \subsection{Affine Supergroup Schemes}
  The situations of affine group schemes and supervector spaces have a common generalization known as {\bf affine supergroup schemes}. Here, we take a {\bf supercommutative Hopf superalgebra} $H = H_0 \oplus H_1$, which is just a commutative Hopf algebra in $\sVec_\KK$. Then, as before, an affine supergroup scheme $G$ arises as $G = \mathrm{Spec} H$, and may be viewed as a functor from the category of supercommutative  algebras to the category of groups. The category $\Rep(G)$ again consists of finite-dimensional $H$-comodules. 
  
  \par
  We can slightly generalize even more. If $G$ is an affine supergroup scheme, suppose that $z \in G(\KK)$ is an element of order $2$ that acts on $\mathcal{O}(G) = H = H_0 \oplus H_1$ by parity (i.e. conjugation by $z$ acts as $(-1)^i$ on $H_i$). Then we can define $\Rep(G, z)$ to be the category of representations of $G$ on superspaces on which $z$ acts by parity.

\begin{Def} A symmetric tensor category $\mathcal C$ is called {\bf Tannakian}
if it is equivalent to ${\rm Rep}(G)$ where
 $G$ is an affine group scheme, and {\bf super-Tannakian}
  if it is equivalent to ${\rm Rep}(G,z)$, where
  $G$ is an affine supergroup scheme.
\end{Def}

\begin{rem} The notion of a super-Tannakian category
subsumes all the previous examples. Namely,
for an affine group scheme $G$ we have ${\rm Rep}(G)={\rm Rep}(G,1)$.
More generally, if $G$ is an affine supergroup scheme then define a new
supergroup scheme $\Bbb Z/2\ltimes G$, where the generator
$z$ of $\Bbb Z/2$ acts by parity. Then ${\rm Rep}(G)={\rm Rep}(\Bbb Z/2\ltimes G,z)$.
For example, ${\rm sVec}_\KK={\rm Rep}(\Bbb Z/2,z)$.
\end{rem}

For more details on these notions and constructions, see \cite{EGNO}, Subsection 9.11.

\subsection{Moderate Growth of Symmetric Tensor Categories and Deligne Categories}
Super-Tannakian categories comprise a large class of symmetric tensor categories, but again, it does not contain all of them, up to equivalence. To identify new symmetric tensor categories which are not of this type, we introduce a notion of size.
\par
If $X \in \mathcal{C}$ is an object in an artinian category $\mathcal{C}$, recall that the {\bf length} of $X$ is the number of composition factors in its Jordan-H\"older series. Now let $\ell_n(X)$ denote the length of $X^{\otimes n}$. If $\mathcal{C} = \Rep(G, z)$, where $G$ is an affine supergroup scheme, then $\ell_n(X)$ grows at most exponentionally in $n$, since $\ell_n(X) \leq (\dim_\KK X)^n$ just by comparing dimensions. In general, if $\ell_n(X)$ of all objects $X$ in a symmetric tensor category $\mathcal{C}$ grow subexponentially, we shall say $\mathcal{C}$ has \textbf{moderate growth}.

\begin{rem}\label{Schur} Recall the Schur-Weyl duality: if ${\rm char}(\KK)=0$
then for any vector space $V$ over $\KK$ and $d\ge 0$,
$$
V^{\otimes d}=\oplus_{\lambda: |\lambda|=d}\pi_\lambda\otimes S^\lambda V,
$$
as $S_d\times GL(V)$-modules, where $\pi_\lambda$ is the irreducible representation
of $S_d$ corresponding to a partition $\lambda$, and $S^\lambda$ is the corresponding Schur functor. It is easy to see that if $\mathcal C$ is of moderate growth
then for every $X\in \mathcal C$ there is a partition $\lambda$
such that the corresponding Schur functor $S^\lambda X$ vanishes
(indeed, otherwise the Schur-Weyl duality implies that
$\dim_\KK {\rm End}(X^{\otimes n})\ge n!$, see Subsection \ref{prope} below). The converse is also true (\cite{D1}), i.e., this property gives another characterization of symmetric tensor categories of moderate growth.
\end{rem}

It turns out that there exist symmetric tensor categories of faster than moderate growth (which are therefore not super-Tannakian). These include the \textbf{Deligne categories}, which are defined as certain interpolation categories of representation categories of various classical groups. For instance, we can interpolate $\Rep GL_n(\C)$ to a category $\Rep GL_t$, where $n \in \mathbb{N}$ and $t \in \C$ (and these categories will be discussed below). However, in characteristic $0$, we have a remarkable theorem of Deligne (2002) which tells us that
for moderate growth there are no further examples:

\begin{thm}\label{Del}(\cite{D1})
  Every symmetric tensor category of moderate growth over an algebraically closed field of characteristic $0$  is super-Tannakian.\end{thm}

\begin{rem}\label{inte} Theorem \ref{Del} implies that dimensions of objects
in a symmetric tensor category of moderate growth over a field of characteristic zero
are integers. This can also be seen directly using Remark \ref{Schur}, as
$\dim S^\lambda X=P_\lambda(\dim X)$, where
$P_\lambda$ is a polynomial with integer roots
given by the Weyl dimension formula for representations of $GL_n(\Bbb C)$.
However, we will see that this property fails for categories of non-moderate growth (e.g., Deligne categories).
\end{rem}

\begin{rem} As we will see later, Theorem \ref{Del} is not true in characteristic $p > 0$.
\end{rem}

\subsection{Tannakian formalism} It turns out that the super-Tannakian category $\Rep(G,z)$ completely determines the pair $(G,z)$ up to an isomorphism. This can be done via the {\bf Tannakian formalism}, discussed in \cite{DM} (see also \cite{EGNO}, Subsection 5.4).

Let us start with the case of Tannakian categories,
$\mathcal C={\rm Rep}(G)$, where $G$ is an affine group scheme over $\KK$.
Then we have the {\bf forgetful functor} $F: \mathcal{C} \rightarrow \sVec_\KK$, which turns out to be the unique symmetric tensor functor up to a (non-unique) isomorphism (see \cite{DM}), and
$G$ can be reconstructed as the group scheme of
tensor automorphisms of $F$, $G=\underline{\rm Aut}_{\otimes}(F)$.
This means that for any commutative $\KK$-algebra $R$,
$G(R)$ is the group of compatible collections $g$
of $R$-linear automorphisms $g_X$ of $F(X)\otimes_\KK R$
which preserve the tensor structure of $F$.

\begin{example}\label{topspace} Let $X$ be a path-connected Hausdorff topological space, and ${\rm LocSys}(X)$ be the category of locally constant sheaves of finite dimensional $\KK$-vector spaces on $X$ (also called local systems). Then ${\rm LocSys}(X)$ is a symmetric tensor category over $\KK$ in a natural way.
Moreover, this category is Tannakian:
given $x\in X$, it is well known that we have a symmetric tensor equivalence
$\bold F_x: {\rm LocSys}(X)\to \Rep_\KK(\pi_1(X,x))$ of ${\rm LocSys}(X)$ with the category
of representations of the fundamental group $\pi_1(X,x)$, which takes the fiber of the local system at $x$ viewed as a representation of $\pi_1(X,x)$. Composing this equivalence with the forgetful functor, we obtain a symmetric tensor functor $F_x:  {\rm LocSys}(X)\to {\rm Vec}_\KK$ which also takes the fiber of the local system at $x$, but now viewed just as a vector space. These functors for different $x$ are isomorphic as stated above, but not canonically: an isomorphism $h_\gamma: F_x\to F_y$ is defined by a homotopy class of paths $\gamma: [0,1]\to X$ with $\gamma(0)=x$ and $\gamma(1)=y$, namely it is the holonomy
along $\gamma$.

The group $\pi_1(X,x)$ is not an affine group scheme, so it cannot be recovered from $F_x$. But we can recover the proalgebraic completion $\widehat{\pi_1(X,x)}_\KK$. Namely, $\widehat{\pi_1(X,x)}_\KK=\underline{\rm Aut}_{\otimes}(F_x)$. We see, however, that since there is no canonical choice of the point $x$, there is no canonical choice of $F_x$, so the group $\widehat{\pi_1(X,x)}_\KK$ is determined by
the category ${\rm LocSys}(X)$ not functorially but only up to (inner) automorphisms.
\end{example}

Motivated by Example \ref{topspace}, a symmetric tensor functor $F: \mathcal C\to {\rm Vec}_\KK$ is called a {\bf fiber functor}.

We have the same story in the more general super-case. Namely, given
a super-Tannakian category $\mathcal C\cong \Rep(G,z)$ we have a unique (up to a non-unique isomorphism) symmetric tensor functor $F: \mathcal C\to {\rm sVec}_\KK$
called a {\bf superfiber functor} (namely, the forgetful functor
on ${\rm Rep}(G,z)$). Moreover, the following theorem holds.

\begin{thm} (Super-Tannakain reconstruction theorem, \cite{D1}) We have
$G\cong \underline{\rm Aut}_{\otimes}(F)$ as supergroup schemes.\footnote{This has the same meaning as in the Tannakian case, except we now need to take $R$ to be a supercommutative $\KK$-superalgebra.}
Moreover, $z\in G(\KK)$ is just the parity automorphism.
\end{thm}

\begin{rem} One may ask how to recognize Tannakian categories among symmetric tensor categories of moderate growth over a field of characteristic zero (which by Deligne's theorem is the same as super-Tannakian). It turns out that these are exactly those of them in which $\dim X=0$ implies $X=0$. Indeed, in this case all dimensions
have to be non-negative integers (see Remark \ref{inte}), and if $\dim X=n$
then $\wedge^{n+1}X=0$ (as $\dim \wedge^kX=\binom{\dim X}{k}$). But then the category is Tannakian by the ``baby Deligne theorem" proved in \cite{D4}.
\end{rem}


\section{Lecture 2: Representation theory in non-integral rank}

\subsection{Deligne Categories}
Let us now turn to Deligne categories, which are categories of non-moderate growth. Let us assume the base field is $\KK = \C$ (though any algebraically closed field of characteristic $0$ would do).
\par
The prototypical example is $\Rep GL_t$ ($t \in \mathbb{C}$), which is an interpolation of $\Rep GL_n(\Bbb C)$, and more details can be found in \cite{DM}, Examples 1.26, 1.27, \cite{CW}, and \cite{EGNO}, Subsection 9.12. The construction can be suitably modified for other classical groups like the orthogonal group and the symplectic group.

In order to construct this category, we want to define $\Rep GL_n(\mathbb{C})$ without mention of $GL_n(\mathbb{C})$ or matrices (after all, how does one make sense of a $t \times t$ matrix for complex $t$?).
\par
It is known that $\Rep GL_n(\C)$ is a semisimple category (e.g., this follows from the Schur-Weyl duality). In particular, if $V$ denotes the vector representation of $GL_n(\C)$, then any representation of $GL_n(\C)$ appears as a direct summand of $[r, s] \coloneqq  V^{\otimes r}\otimes V^{* \otimes s}$ for sufficiently large $r, s$. Hence, we can define $\widetilde{\mathcal{C}}$ to be the full subcategory with objects of the form $[r,s]$, and then let $\mathcal{C} = \Rep GL_n(\Bbb C)$ be the Karoubian completion of $\widetilde{\mathcal{C}}$; as already mentioned, this means that we add direct summands and finite direct sums.
\par
Unfortunately, $V$ itself is still defined using matrices, and therefore so are the Hom spaces in $\widetilde{\mathcal C}$. We'd like to abstract that away, as already suggested by the notation $[r, s]$. To that end, we need to redefine the Hom spaces. To begin, we have
\begin{align*}
  \Hom([r,s], [p,q]) = \Hom_{GL_n(\Bbb C)}(V^{\otimes r}\otimes V^{* \otimes s}, V^{\otimes p}\otimes V^{* \otimes q}) \\
  = \Hom_{GL_n(\Bbb C)}(V^{\otimes r}\otimes V^{\otimes q}, V^{\otimes p}\otimes V^{\otimes s}) \\
  = \begin{cases}
    0 & r + q \neq s + p \\
    \C[S_d]/I & n =  r + q = s + p.
\end{cases}
\end{align*}
where $S_d$ is the symmetric group on $\{1, 2, \dots, d\}$ and $I$ is some ideal which is $0$ if $d \leq \dim V$. The second equality is an application of adjunction properties of duals and the third equality is a consequence of the Schur-Weyl duality.

\begin{prop}
  If $\dim V \gg 0$, then the composition map
  \[\Hom([r_1, s_1], [r_2, s_2]) \times \Hom([r_2, s_2], [r_3, s_3]) \rightarrow \Hom([r_1, s_1], [r_3, s_3])\]
  is a polynomial in $n = \dim V$.
\end{prop}

\begin{proof} (sketch) First of all, observe that for large $n$, these Hom spaces are independent of $n$. Namely, a basis for morphisms is given by {\bf walled Brauer diagrams}. In more detail, a basis element of $\Hom([r,s], [p, q])$ can be depicted as follows (see the diagram below). We have two rows of arrows, with $p$ up-arrows and $q$ down-arrows in the top row, and $r$ up-arrows and $s$ down-arrows in the bottom row. Up arrows and down arrows are separated by a wall. A basis element is then a perfect matching connecting disjoint pairs of arrows such that connecting arrows between arrows in different rows do not cross the wall (i.e., the direction of arrows is preserved).

 As an example, a basis element of $\Hom([4,4], [3,3])$ looks like this:

  
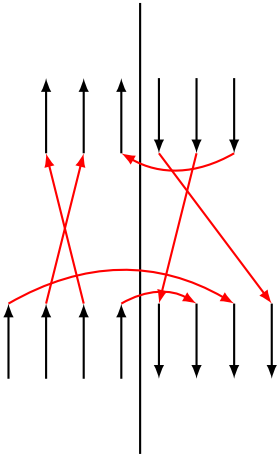
\begin{center}
    \begin{tikzpicture}
    \draw[->, thick, >=latex] (1/2,0) -- (1/2,1);
    \draw[->, thick, >=latex] (2/2,0) -- (2/2,1);
    \draw[->, thick, >=latex] (3/2,0) -- (3/2,1);
    \draw[<-, thick, >=latex] (4/2,0) -- (4/2,1);
    \draw[<-, thick, >=latex] (5/2,0) -- (5/2,1);
    \draw[<-, thick, >=latex] (6/2,0) -- (6/2,1);
    \draw[-, thick] (7/4, -4) -- (7/4, 2);
    \draw[->, thick, >=latex] (0/2,-3) -- (0/2,-2);
    \draw[->, thick, >=latex] (1/2,-3) -- (1/2,-2);
    \draw[->, thick, >=latex] (2/2,-3) -- (2/2,-2);
    \draw[->, thick, >=latex] (3/2,-3) -- (3/2,-2);
    \draw[<-, thick, >=latex] (4/2,-3) -- (4/2,-2);
    \draw[<-, thick, >=latex] (5/2,-3) -- (5/2,-2);
    \draw[<-, thick, >=latex] (6/2,-3) -- (6/2,-2);
    \draw[<-, thick, >=latex] (7/2,-3) -- (7/2,-2);

    \draw [->, thick, >=latex, red] (0/2,-2) to [bend left] (6/2,-2);
    \draw [->, thick, >=latex, red] (1/2,-2) to  (2/2,0);
    \draw [->, thick, >=latex, red] (2/2,-2) to  (1/2,0);
    \draw [->, thick, >=latex, red] (3/2,-2) to [bend left] (5/2,-2);
    \draw [->, thick, >=latex, red] (4/2,0) to  (7/2,-2);
    \draw [->, thick, >=latex, red] (5/2,0) to  (4/2,-2);
    \draw [->, thick, >=latex, red] (6/2,0) to [bend left] (3/2,0);

    \end{tikzpicture}
  \end{center}

Now, let's consider the composition of morphisms.
Composition is given by stacking two diagrams and contracting any loops that arise, while introducing a factor of $n$ for each loop contracted. Thus the structure constants of composition are polynomial (even monomial!) functions of $n$, as claimed.
\end{proof}

So now we can interpolate to define the category $\widetilde{\Rep} \ GL_t$ as the category whose objects are $[r, s]$ for $r, s \in \Z_{\geq 0}$ and the morphisms and their composition are as above, but the parameter $n$ is replaced by $t$.
This is a symmetric monoidal category, with tensor product given just by placing walled Brauer diagrams next to each other (so $[p,q]\otimes [r,s]=[p+r,q+s]$), and the symmetry is given by swapping the diagrams.
Moreover, this category is rigid, with $[r,s]^*=[s,r]$, and its unit object is $[0,0]$.

The endomorphism rings in the category $\widetilde{\Rep} \ GL_t$, namely= $W_{r,s}(t) \coloneqq \End ([r, s])$, are called {\bf walled Brauer algebras}; we see that as a vector space, $W_{r,s}(t)$ is isomorphic to $\mathbb{C}[S_{r+s}]$ but it has a different multiplication depending on $t$.

\begin{prop}\label{semi} (\cite{CDDM}, Theorem 6.3) If $t\notin \Bbb Z$ then
  the algebras $W_{r,s}(t)$ are semisimple.
  \end{prop}

We omit the proof but consider an example. In the case $r = s = 1$, a basis of $W_{1, 1}(t)$ consists of two diagrams:
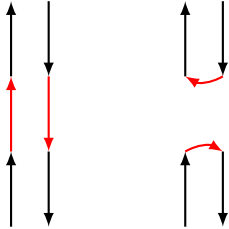
\begin{center}
  \begin{tikzpicture}

  \draw[->, thick, >=latex] (3/2,0) -- (3/2,1);
  \draw[<-, thick, >=latex] (4/2,0) -- (4/2,1);

  \draw[->, thick, >=latex] (3/2,-2) -- (3/2,-1);
  \draw[<-, thick, >=latex] (4/2,-2) -- (4/2,-1);

  \draw [->, thick, >=latex, red] (3/2,-1) to (3/2, 0);
  \draw [<-, thick, >=latex, red] (4/2,-1) to  (4/2,0);

  \end{tikzpicture} \qquad\qquad
  \begin{tikzpicture}

  \draw[->, thick, >=latex] (3/2,0) -- (3/2,1);
  \draw[<-, thick, >=latex] (4/2,0) -- (4/2,1);

  \draw[->, thick, >=latex] (3/2,-2) -- (3/2,-1);
  \draw[<-, thick, >=latex] (4/2,-2) -- (4/2,-1);

  \draw [->, thick, >=latex, red] (3/2,-1) [bend left] to (4/2, -1);
  \draw [->, thick, >=latex, red] (4/2,0) [bend left] to  (3/2,0);

  \end{tikzpicture}
\end{center}

 The diagram on the left is the identity element, and the diagram on the right, which we will call $a$, satisfies the relation $a^2 = ta$. Hence, $W_{1,1}(t) = \C[a]/(a^2-ta)$, which is semisimple if and only if $t \neq 0$.
\par
We then define the category $\Rep GL_t$ for $t \in \mathbb{C}$ to be the Karoubian completion of $\widetilde{\Rep} \ GL_t$.

Proposition \ref{semi} implies

\begin{cor} If $t\notin \Bbb Z$ then the category $\Rep GL_t$ is a semisimple symmetric tensor category.
\end{cor}

\subsection{Properties of $\Rep GL_t$ for $t\notin \Bbb Z$}\label{prope}
We can now show that $\Rep GL_t$ does not have moderate growth, and therefore is not super-Tannakian. If $X = \bigoplus_i m_i X_i$ is the direct sum of simple objects, then its length $\ell(X)$ equals $\sum_i m_i$. It follows that $\End X = \bigoplus_i {\rm Mat}_{m_i}$ is a direct sum of matrix rings, so $\dim\End X = \sum_i m_i^2$. Therefore, $\ell(X) \geq \sqrt{\dim \End X}$. In our particular case, if we take $X = V^{\otimes d} = [1, 0]^{\otimes d}$, then $\End X = \End V^{\otimes d} = \mathbb{C}[S_d]$, so
$$
\ell_d(V)=\ell(X) \geq \sqrt{\dim \End X} = \sqrt{d!},
$$
 which grows faster than exponential.

 Another way to show that this category is not super-Tannakian is to compute categorical dimensions. We have $\dim V = \tr(1_V) = t \not \in \Z$, while in a super-Tannakian category $\dim V$ has to be an integer.
\par
Since for $t \not \in \Z$, the category $\Rep GL_t$ is semisimple, it is determined, as an abelian category, by describing its simple objects. The simple objects in $\Rep GL_t$ for $t \not \in \Z$ are $V_{\lambda, \mu}$ indexed by pairs of partitions  $\lambda = (\lambda_1, \dots, \lambda_r), \mu = (\mu_1, \dots, \mu_s)$ (with no restrictions), which interpolates the irreducible representations of $GL_n(\Bbb C)$ whose highest weights are $(\lambda_1, \dots, \lambda_r, 0, \dots, 0, -\mu_s, \dots -\mu_1)$, where there are $n - r - s$ zeros (this makes sense for $n \geq r + s$).
\par
\subsection{Tensor ideals, quotient categories, and negligible morphisms} We would now like to take a closer look at the case when $t \in \Z$. In this case the category $\Rep GL_t$ is not a symmetric tensor category but rather is only pseudotensor. To analyze this case, we will need the notion of a {\bf tensor ideal} in a symmetric pseudotensor category (as well as that of the {\bf quotient} by such an ideal), and the notion of a {\bf negligible morphism}.

\begin{Def}[Tensor Ideal] Let $\mathcal C$ be a symmetric pseudotensor category.
  A tensor ideal $I$ in $\mathcal{C}$ is a collection of subspaces
  $$
  I = \{I(X, Y) \subseteq \Hom_{\mathcal{C}}(X,Y)\}_{X, Y \in \mathcal{C}}
  $$
   such that for all objects $X, Y, Z, T \in \mathcal{C}$,
  \begin{enumerate}
    \item $\forall \alpha \in I(X, Y), \beta \in \Hom_{\mathcal{C}}(Y, Z), \gamma \in  \Hom_{\mathcal{C}}(Z, X)$, we have $\beta \circ \alpha \in I(X, Z)$ and $\alpha \circ \gamma \in I(Z, Y)$.
    \item $\forall \alpha \in I(X, Y), \beta \in \Hom_{\mathcal{C}}(Z, T)$, we have $\alpha \otimes \beta \in I(X \otimes Z, Y \otimes T)$ and $\beta \otimes \alpha \in I(Z \otimes X, T \otimes Y)$.
  \end{enumerate}
\end{Def}

Thus, the notion of a tensor ideal is, in a sense, a categorification of the notion of an ideal in a ring.

If $I\subset \mathcal C$ is a tensor ideal, then one can define a symmetric pseudotensor category $\mathcal C/I$, which has the same
objects as $\mathcal C$ but ${\rm Hom}_{\mathcal C/I}(X,Y)={\rm Hom}_{\mathcal C}(X,Y)/I(X,Y)$. This category is called the {\bf quotient of $\mathcal C$ by $I$}, and it categorifies the notion of the quotient of a ring by an ideal.

\begin{Def}[Negligible Morphism]
A morphism $f : X \rightarrow Y$ in a symmetric pseudotensor category $\mathcal{C}$ is said to be \textbf{negligible} if for all $g : Y \rightarrow X$, ${\rm Tr}(f \circ g) = 0$.
\end{Def}

\begin{prop} The collection $\mathcal N(\mathcal{C})$ of negligible morphisms in $\mathcal{C}$ forms a tensor ideal.
\end{prop}

\begin{proof}
This follows by straightforward verification using the properties of the trace.
\end{proof}

\subsection{Properties of $\Rep GL_t$ when $t\in \Bbb Z$}
Now let us study the category $\Rep GL_t$ for $t\in \Bbb Z$. We will restrict to the case $t\ge 0$.\footnote{It is easy to see that $\Rep GL_{-t}$ is equivalent to $\Rep GL_t$ with a modified symmetric structure, so the case $t\in \Z_{\le 0}$ is completely parallel.} In this case, we have the following result.

Let $GL_{m|n}(\Bbb C)$ be the {\bf general linear supergroup} (see \cite{EHS} and references therein).

\begin{thm} (\cite{C,CW})
Suppose $t \in \Z_{\geq 0}$. There exists an infinite chain of nontrivial tensor ideals $I_0 \supset I_1\supset I_2\cdots$ in $\Rep GL_t$, where $I_0=\mathcal N(\Rep GL_t)$ is the ideal of negligible morphisms. Moreover, we have an embedding  $\Rep GL_t/I_m\hookrightarrow \Rep GL_{t + m | m }(\Bbb C)$ as a full pseudotensor subcategory, which is an equivalence $\Rep GL_t/I_0\cong\Rep GL_{t }(\Bbb C)$ when $m = 0$. 
\end{thm}

The upshot is that the Deligne categories $\Rep GL_t$ with integer $t$ play a role in the representation theory of supergroups.

\subsection{The universal property and the abelian envelope of $\Rep GL_t$}
The category $\Rep GL_t$ has an important universal property which easily follows from the definition.

\begin{prop}\label{unipro} For any $t\in \Bbb C$, additive symmetric monoidal functors $F: \Rep GL_t \rightarrow \mathcal{C}$ to a Karoubian symmetric monoidal category $\mathcal{C}$ correspond to rigid objects in $\mathcal{C}$ of dimension $t$.   The correspondence is given by $F \leftrightarrow F([1, 0])$.
\end{prop}

Let us now discuss the abelian envelope of $\Rep GL_t$. We have seen that for $t \notin \Bbb Z$, the categories $\Rep GL_t$ are symmetric tensor categories (i.e., abelian), while for $t \in \Z$ they are only pseudotensor (i.e., Karoubian). We have also seen that in the latter case (when $t\ge 0$) we can quotient $\Rep GL_t$ by a tensor ideal $I_m$ and then embed the quotient into an (abelian) symmetric tensor category $\Rep GL_{t+m|m}(\Bbb C)$. This invites a natural question: can we embed $\Rep GL_t$ into an abelian symmetric tensor category {\bf faithfully}, i.e. without quotienting out any morphisms? And if yes, does there exist the universal such embedding, i.e., the {\bf abelian envelope} of $\Rep GL_t$? \footnote{Note that by the 2-categorical Yoneda lemma, such abelian envelope is unique if exists.}
\par
This was conjectured by Deligne in \cite{D2} and the answer turns out to be ``yes". This is proved in \cite{EHS} by constructing the abelian envelope of $GL_t$ as a suitable limit category, denoted $\Rep^{ab} GL_t$. More specifically, $\Rep^{ab} GL_t$ is a certain limit of $\Rep GL_{t+m|m}(\Bbb C)$ as $m \rightarrow \infty$. This envelope satisfies a universal property similar to one of Proposition \ref{unipro}: faithful additive symmetric monoidal functors $F:\Rep^{ab} GL_t \rightarrow \mathcal{C}$, where $\mathcal{C}$  is a symmetric tensor category, correspond to objects $V \in \mathcal{C}$ of dimension $t$ such that the Schur functors $S^{\lambda}V \neq 0$ for all partitions $\lambda$. Another general approach to constructing abelian envelopes which allows one to construct
$\Rep^{ab} GL_t$ is developed in \cite{Co}.

\subsection{Ultrafilters}
Ultrafilters are a tool from model theory that in particular is very useful for working with various interpolation categories both in zero and positive characteristic. It was used to study such categories by Deligne in \cite{D2} and then \cite{D3}, see also \cite{H}, 3.3.
For a review of ultrafilters and how they can be used to study interpolation categories see \cite{EKR,K} and references therein.

\begin{Def}[Ultrafilter]
  An {\bf utlrafilter} on a non-empty set $S$ is a collection $\mathcal{F}$ of subsets of $S$ satisfying the following axioms:

  \begin{enumerate}
    \item $\forall X \subset S$, exactly one of $X, X^c$ is in $\mathcal{F}$;
    \item if $X \in \mathcal{F}$ and $X \subset Y$, then $Y \in \mathcal{F}$;
    \item if $X, Y \in \mathcal{F}$, then $X \cap Y \in \mathcal{F}$.
  \end{enumerate}
\end{Def}

In particular, (1),(2) imply that 
$S \in \mathcal{F}$. 

In other words, an ultrafilter $\mathcal F$ on $S$ is the same thing as a unital ring homomorphism
$\chi_{\mathcal F}: {\rm Fun}(S,\Bbb F_2)\to \Bbb F_2$, or, equivalently, a maximal (=prime) ideal in ${\rm Fun}(S,\Bbb F_2)$. Namely, denoting by $1_X$ the indicator function for $X\subset S$, we have $X\in \mathcal{F}$ iff $\chi_{\mathcal F}(1_X)=1$. Indeed, axioms (1)-(3) take the following form in terms of $\chi=\chi_{\mathcal F}$: 

 \begin{enumerate}
    \item $\chi(a)+\chi(1-a)=1$; 
    \item If $\chi(a)=1$ and $ab=a$ then $\chi(b)=1$;
    \item If $\chi(a)=\chi(b)=1$ then $\chi(ab)=1$. 
  \end{enumerate}

  Clearly, any unital homomorphism satisfies these properties, and vice versa.\footnote{Indeed, let $\chi(a)=\chi(b)=1$. Then by (3)
$\chi(ab)=1$. But all elements of ${\rm Fun}(S,\Bbb F_2)$ are idempotents. Thus $ab(1-a-b)=ab$. So by (2), $\chi(1-a-b)=1$. Thus by (1), 
$\chi(a+b)=0$. Replacing $a$ by $1-a$ and/or $b$ with $1-b$, we thus get  
$\chi(a+b)=\chi(a)+\chi(b)$ for all $a,b$. Similarly, from (3) we get in the same way that $\chi(ab)=\chi(a)\chi(b)$ for all $a,b$, i.e., $\chi$ is a ring homomorphism which is unital by (1), as claimed.} 

The intuitive way to understand this definition is that we designate subsets of $S$ as consisting of {\bf almost everything} if they lie in $\mathcal{F}$ or {\bf almost nothing} otherwise.

A simple example is a {\bf principal ultrafilter}: if $s \in S$, we can define the ultrafilter $\mathcal{F}_s = \{X \subseteq S | s \in X\}$, with $\chi_{\mathcal F_s}(a)=a(s)$. These are boring, but clearly are the only possible ultrafilters for {\bf finite} sets. In contrast, Zorn's lemma (i.e., axiom of choice) implies that any {\bf infinite} set $S$ admits a {\bf non-principal ultrafilter}. Since its existence is proved non-constructively,\footnote{Namely, let ${\rm Fun}_0(S,\Bbb F_2)\subset {\rm Fun}(S,\Bbb F_2)$ be the ideal of functions with finite support, proper when $S$ is infinite. Then non-principal ultrafilters on $S$ correspond to maximal ideals in ${\rm Fun}(S,\Bbb F_2)$ containing ${\rm Fun}_0(S,\Bbb F_2)$ (and such ideals exist by Zorn's lemma).}
there can be no explicit construction of such an ultrafilter. However, the mere existence of $\mathcal F$ enables various powerful constructions using the so-called {\bf ultraproduct} construction discussed below.

Given a non-principal ultrafilter $\mathcal F$ on $S$, we say
that a certain statement holds for {\bf almost all} $x\in S$ with respect to $\mathcal F$ if it holds for all $x\in X$ where $X\in \mathcal F$. It is easy to see
that complements of finite sets are all in $\mathcal{F}$, so
if the statement holds for almost all $x\in S$ in the usual sense
then it does so with respect to $\mathcal F$, but not vice versa, in general.
In this section, when we say ``for almost all $x$'' we will always mean ``for almost all $x$ with respect to a fixed $\mathcal F$".

\begin{Def}[Ultraproduct]
Let $\mathcal{F}$ be a non-principal ultrafilter on $\mathbb{N} = \{1, 2, \dots\}$, and let $X_1, X_2, \dots $ be a sequence of sets. The \textbf{ultraproduct} with respect to the ultrafilter $\mathcal{F}$ is denoted as $\prod_{\substack{i \in \mathbb{N} \\ \mathcal{F}}} X_i$. The elements of the ultraproduct are defined as equivalence classes of sequences $\{x = (x_1, x_2, \dots)\}$ with $x_i \in X_i$ (similar to the usual product of the $X_i$), with the caveat that $x$ only needs to be defined for almost all (not necessarily all) $i$, and the equivalence relation is given by $x \sim x'$ if $x_i = x_i'$ for almost all $i$.
\end{Def}
As an example, if $X$ is a finite set, then $\prod_{\substack{i \in \mathbb{N} \\ \mathcal{F}}} X = X$, as we can take the equivalence class representatives to be $s_x = (x, x, x, \dots )$ for each $x \in X$. Indeed, if $s$ is some element in the ultraproduct, then let $I_x(s)$ denote the set of indices at which $x\in X$ appears. Then, exactly one of the $I_x(s)$'s must be in $\mathcal{F}$. This is not true, however, for an infinite set $X$; in this case the ultraproduct contains $X$ but is much larger.
\par
A very useful property of the ultraproduct is that if almost all $X_i$ share a structure or property expressible in terms of {\bf first-order logic} (informally, this means that it does not use the notion of {\bf size}, in particular, {\bf finiteness}), then it is inherited by  $\prod_{\substack{i \in \mathbb{N} \\ \mathcal{F}}} X_i$. For instance, ultraproducts of groups/rings/fields are groups/rings/fields, respectively (with pointwise addition/multiplication).
\par
For example, consider  $\prod_{\substack{i \in \mathbb{N} \\ \mathcal{F}}} \overline{\mathbb{Q}}$. This ultraproduct will be an algebraically closed field of characteristic $0$ (as these are first order properties/structures shared by all the factors) with cardinality of the continuum. But by Steinitz's theorem, any such field is isomorphic (non-canonically) to $\mathbb{C}$. Thus $\prod_{\substack{i \in \mathbb{N} \\ \mathcal{F}}}\overline{\mathbb{Q}}\cong \Bbb C$.
We can also consider a non-principal ultrafilter $\mathcal{F}$ on the set of all primes, and define the ultraproduct $\prod_{\substack{p \ \mathrm{prime} \\ \mathcal{F}}} \overline{\mathbb{F}_p}$ of algebraically closed fields of characteristic $p$. This will again be a field of characteristic $0$ of cardinality of the continuum. Indeed for every $p$ all but one (thus, almost all) factors have the first order property of having characteristic $\ne p$, hence so does the ultraproduct. Thus we again have a non-canonical isomorphism
$\prod_{\substack{p \ \mathrm{prime} \\ \mathcal{F}}} \overline{\mathbb{F}_p}=\Bbb C$.
The same holds for the ultraproduct over any infinite set of primes with respect to some non-principal ultrafilter on this set.

\subsection{The ultraproduct construction of $\Rep GL_t$ for transcendental $t$}
We can similarly define ultraproducts of (essentially small) categories.

For example, consider   $ \mathcal{C} \cong \prod_{\substack{n \in \mathbb{N} \\ \mathcal{F}}} \Rep GL_n(\overline{\mathbb{Q}})$. This is a rigid symmetric monoidal category linear over the field $\prod_{\substack{i \in \mathbb{N} \\ \mathcal{F}}} \overline{\mathbb{Q}} \cong \mathbb{C}$. However, this is not a symmetric tensor category, as, for instance, the artinian condition may fail (indeed, this is not a first order condition, as it requires objects to have finite length!). Similarly, Hom spaces in this category need not be finite dimensional.

 Nevertheless, we can take a suitable artinian subcategory with finite dimensional Hom spaces as follows. Let $V = (V_1, V_2, \dots )$, where $V_n$ is the natural representation of $GL_n(\overline{\mathbb{Q}})$. Then $V$ is an object in $\mathcal{C}$, so let $\langle V\rangle$ be the full subcategory tensor generated by $V$ (i.e., the full subcategory comprised by subquotients of direct sums of tensor products of $V$ and $V^*$). This subcategory 
 is artinian, since the dimension of 
 ${\rm Hom}_{
 GL_n(
 {\mathbb{\bar Q})}}
 ((\overline{\Bbb Q}^n)^{\otimes r},
 (\overline{\Bbb Q}^n)^{\otimes s})
 $
 stabilizes as $n\to \infty$. 
 Moreover, we have the following theorem due to Deligne, \cite{D2}:

\begin{thm}
  The category $\langle V\rangle$ is equivalent to $\Rep GL_t$, where $t$ is the image in $\Bbb C$ of the element  $(1, 2, 3, \dots) $ of the ultraproduct $\prod_{\substack{i \in \mathbb{N} \\ \mathcal{F}}} \overline{\mathbb{Q}}$ under the isomorphism $\xi: \prod_{\substack{i \in \mathbb{N} \\ \mathcal{F}}} \overline{\mathbb{Q}} \cong \mathbb{C}$.
\end{thm}
Note that $t$ is a transcendental number, since there is no nontrivial algebraic equation over $\Bbb Q$ satisfied by infinitely many positive integers. Note also that $t$ depends on the isomorphism $\xi$; indeed, if $\xi$ is replaced by its composition with an automorphism
of $\Bbb C$ sending $t$ to $t'$ then $t$ is replaced by $t'$. It is well known that such an automorphism exists for any $t'$, which shows that for any fixed $\mathcal F$
we obtain a new construction of $\Rep GL_t$ for all transcendental $t\in \Bbb C$.

\subsection{Generalization to algebraic $t$}
This can also be done for algebraic $t\in \Bbb C$, but then we have to resort to characteristic $p$ (see \cite{H2}). Namely, let $q$ be the minimal (monic) polynomial of $t$, so $q$ is irreducible over $\Bbb Q$ and $q(t) = 0$. We will need the following well known and simple lemma from elementary number theory:

\begin{lem}
  There exist infinitely many primes $p$ such that there exists an integer $n_p \in \mathbb{N}$ satisfying $q(n_p) = 0$ modulo $p$.
\end{lem}

Denote the set of such primes $\Bbb S_q$ and for $p\in \Bbb S_q$ choose $n_p$ so that $n_p\to \infty$ as $p\to \infty$ in $\Bbb S_q$. The set $\Bbb S_q$ admits a non-principal ultrafilter $\mathcal{F}$. Consider the ultraproduct $\prod_{\substack{p \in \Bbb S_q \\ \mathcal{F}}} \overline{\mathbb{F}_p}$, which as we know is isomorphic to $\mathbb{C}$. Now, the sequence $(n_p, p\in \Bbb S_q)$  (where $n_p$ is reduced mod $p$) is mapped to some complex number $s \in \mathbb{C}$, which this time is not transcendental but rather is a root $q$ (as this is so for all $p\in \Bbb S_q$). Since the automorphisms of $\mathbb{C}$ act transitively on the roots of $q$ over $\mathbb{Q}$, the isomorphism $\prod_{\substack{p \in \Bbb S_q \\ \mathcal{F}}} \overline{\mathbb{F}_p}\cong \Bbb C$
 can be chosen so that $s=t$.
\par
Therefore, if $t \not \in \Z$ then $\Rep GL_t$ is the subcategory of the ultraproduct tensor generated by $V$:
$$
\Rep GL_t = \langle V \rangle \subseteq  \prod_{\substack{p \in \Bbb S_q \\ \mathcal{F}}} \Rep GL_{n_p}(\overline{\mathbb{F}_p}),
$$
 where $V = (V_{n_p}, p\in \Bbb S_q)$ is the collection of vector representations, similar to above. On the other hand, if $t\in \Bbb Z$ (in which case $\Bbb S_q$ is the set of all primes and we can take $n_p=p+t$), one can show using the methods of \cite{H2} that $\langle V\rangle$ is an artinian category with finite dimensional Hom spaces, hence a symmetric tensor category (namely, we have to show that the length of $V_p^{\otimes k}\otimes V_p^{*\otimes l}$ is bounded as $p\to \infty$ for fixed $k,l$, which can be checked by methods of modular representation theory). Then the results of \cite{EHS} imply that
 $\langle V\rangle$ is nothing but the familiar abelian envelope $\Rep^{\rm ab} GL_t$.
So this gives another construction of this abelian envelope.

\subsection{Generalization to Positive Characteristic}
We can perform similar constructions in positive characteristic, but the story is going to be trickier because the walled Brauer algebra is no longer semisimple, so the diagrammatic construction we used in characteristic zero will not produce a symmetric tensor category (only a pseudotensor one). This is where the ultraproduct construction will really shine.

We start by fixing a prime $p$ and considering the ultraproduct $\prod_{\substack{n \in \mathbb{N} \\ \mathcal{F}}} \Rep GL_n(\overline{\mathbb{F}_p})$. This is a symmetric tensor category over $\KK = \prod_{\substack{n \in \mathbb{N} \\ \mathcal{F}}} \overline{\mathbb{F}_p}$, which is some (uncountable) algebraically closed field of characteristic $p$.\footnote{In fact, as explained in \cite{D3}, the whole construction can be done over $\Bbb F_p$, which produces a category also over $\Bbb F_p$. Then we can tensor with any algebraically closed field of characteristic $p$. So the fact that $\KK$ is uncountable is not important here.} The object $V=(V_1,V_2,...)$ consisting of vector representations generates a full symmetric tensor subcategory $\langle V\rangle$ in this ultraproduct (\cite{D3,H}), and $\dim V=t_0=(1,2,3,...)\in \Bbb K$. Namely, $t_0\in \Bbb F_p$ is the value taken by $n$ modulo $p$ for almost all $n$; it depends on the choice of $\mathcal F$.

It turns out that, unlike the characteristic zero case,  even if the dimension $t_0$ of $V$ is fixed, the category $\langle V\rangle$ is still not uniquely determined and depends on the choice of the ultrafilter $\mathcal{F}$. At first sight this is bad news, since $\mathcal{F}$ is a quintessentially non-constructive object. It turns out, however, that this dependence can be expressed in terms of some explicit invariants of $\mathcal{F}$, so that the nonconstructive nature of $\mathcal{F}$ is not really an issue.
To describe these invariants, we need to discuss the differences between
dimensions of objects in symmetric tensor categories in zero and positive characteristic.

In characteristic zero, as we have seen on the example of Deligne categories, the dimension of an object $X$ can be any element $t\in \KK$. However, once this
number is known, the dimensions of exterior powers $\wedge^n X$ are all determined
from the formula
$$
\dim \wedge^n X=\binom{t}{n}=\frac{t(t-1)...(t-n+1)}{n!}.
$$
The characteristic $p$ situation differs from this in two important respects. First of all, the dimension
can no longer be any element in $\KK$. In fact, we have the
following lemma.

\begin{lem}\label{infp}
If $\mathcal{C}$ is a symmetric tensor category\footnote{This Lemma holds, with the same proof, for any symmetric pseudotensor category in which the trace of any nilpotent endomorphism is zero.} over a field $\KK$ of characteristic $p > 0$, then for any object $X \in \mathcal{C}$ its dimension $\dim X$ lies in the prime subfield $\mathbb{F}_p$ of $\KK$.  
\end{lem}

\begin{proof} Let $\sigma: X^{\otimes p}\to X^{\otimes p}$ be the cyclic permutation.
Then $\sigma^p=1$, hence  $(1-\sigma)^p=0$. Thus the morphism $A:=1-\sigma: X^{\otimes p}\to X^{\otimes p}$ is nilpotent. Hence ${\rm Tr}(A)=0$ (as there is a filtration of $X^{\otimes p}$ by kernels of powers of $A$ which is strictly preserved by $A$). But on
the other hand ${\rm Tr}(A)={\rm Tr}(1)-{\rm Tr}(\sigma)=d^p-d$, where $d=\dim X$.
Thus $d^p-d=0$, i.e., $d\in \Bbb F_p$.
\end{proof}

The second difference is that binomial coefficients $\binom{d}{n}$
do not make sense for $n\ge p$, and in fact the dimensions of $\wedge^n X$
are not determined by $\dim X$. More precisely, the situation is as follows (see \cite{EHO}). Let $\dim X=t_0$. Then for
$n<p$ we have the usual formula
$\dim \wedge^n X=\binom{t_0}{n}$,
but for $n=p$ this formula no longer makes sense because of vanishing denominator, and in fact $\dim\wedge^p X$ is a new parameter $t_1$, independent of $t_0$,
which by Lemma \ref{infp} also belongs to $\Bbb F_p$.
The numbers $t_0$, $t_1$ then determine
$\dim \wedge^n X$ for $n<p^2$, but
$\dim \wedge^{p^2} X=t_2\in \Bbb F_p$
is a new parameter independent on $t_0,t_1$, and so on.
The sequence $t_0,t_1,t_2...$ can be viewed as the sequence
of digits of a $p$-adic integer $t=....t_2t_1t_0\in \Bbb Z_p$, which is called the
{\bf $p$-adic dimension} of $X$, denoted ${\rm Dim}(X)$.
Moreover, it is shown in \cite{EHO} that the dimensions
of all exterior powers of $X$ can be computed by the formula
$$
\sum_{n=0}^\infty \dim (\wedge^n X) z^n=(1+z)^t,
$$
where
$$
(1+z)^t:=(1+z)^{t_0}(1+z^p)^{t_1}(1+z^{p^2})^{t_2}...
$$
The motivation for this notation is that the map
$t\mapsto (1+z)^t$ is a continuous group homomorphism
$\Bbb Z_p\to \Bbb F_p[[z]]^\times$.

Now we can come back to Deligne categories.

\begin{thm} (\cite{D3,H})
The category $\langle V\rangle$ depends only the $p$-adic integer $t$ (and is otherwise independent of $\mathcal F$).
\end{thm}

The category $\langle V\rangle$ is therefore denoted $\Rep^{\rm ab} GL_t$.

In fact, we can recover $t$ from the ultrafilter $\mathcal F$ without mentioning any categories whatsoever. Namely, recall that if $\bold X$ is a compact metric space and $x_1, x_2, \dots$ is a sequence in $\bold X$, then by the Bolzano-Weierstrass theorem it has a convergent subsequence. Therefore, this sequence has subsequential limits forming some non-empty closed subset $E \subseteq \bold X$. The choice of a non-principal ultrafilter $\mathcal F$ on $\Bbb N$ will ``vote" for one of them, $x\in E$, in the sense that almost all $x_n$ with respect to $\mathcal{F}$ belong to any given neighborhood of $x$. Therefore, any sequence in $X$ has a limit
$$
x = \lim_{\substack{n\rightarrow \infty \\ \mathcal{F}}} x_n
$$
with respect to the ultrafilter $\mathcal F$, even though it may not have a limit in the usual sense.

In particular, we can look at the sequence $(1, 2, 3, \dots)$ in $\mathbb{Z}_p$ (the $p$-adic integers, which form a compact metric space under $p$-adic norm), which is dense in $\mathbb{Z}_p$; hence a choice of ultrafilter will pick one (potentially arbitrary) element of $\Bbb Z_p$ as the limit, and this element is precisely $t$.

\begin{rem} It is expected that the category $\Rep^{\rm ab} GL_t, t\in \Bbb Z_p$ is universal (in an appropriate sense) symmetric tensor category with an object of
$p$-adic dimension $t$, but such a result is not
available thus far, even as a precise conjecture.
\end{rem} 

\section{Lecture 3: Symmetric tensor categories of moderate growth and modular representation theory}

\subsection{Semisimplification of Symmetric Tensor Categories}
As we have explained, by Deligne's theorem symmetric tensor categories of moderate growth over a field of characteristic $0$ are super-Tannakian. In contrast, in characteristic $p > 0$, we can construct counterexamples to this theorem. This is interesting because we can do linear algebra, representation theory, and algebraic geometry in such a category, which will differ in an essential way from the usual setting.
\par
In order to construct such an example, we first review a general procedure known as {\bf semisimplification of symmetric pseudotensor categories} (following  \cite{EO} and references therein). As the name suggests, this process involves starting with an arbitrary symmetric tensor (or, more generally, pseudotensor) category and constructing its semisimplification (i.e., a new semisimple symmetric tensor category) by quotienting the Hom spaces by the subspaces of negligible morphisms. Loosely speaking, in this way we force Schur's lemma to hold for indecomposable objects. More precisely, we make the following definition.

Let $\mathcal{C}$ be a symmetric pseudotensor category in which the trace of any nilpotent endomorphism is zero (for example, this is true if $\mathcal{C}$ admits a symmetric tensor functor into a symmetric tensor category, as then the trace can be computed in this target category).

\begin{Def}(Semisimplification).
We define the {\bf semisimplification} $\overline{\mathcal{C}} = \mathcal{C}/\mathcal N(\mathcal{C})$ to be the category whose objects are the same as those of $\mathcal{C}$ and whose morphisms are given by $\Hom_{\overline{\mathcal{C}}}(X, Y) = \Hom_{\mathcal{C}}(X, Y)/\mathcal N(X, Y)$, where $\mathcal N(X, Y)$ is the space of negligible morphisms from $X$ to $Y$.
\end{Def}

Since $\mathcal N(\mathcal C)$ is a tensor ideal,
the category $\overline{\mathcal C}$ is a symmetric pseudotensor category.
Moreover, it is equipped with the symmetric monoidal {\bf semisimplification functor}
$S: \mathcal C\to \overline{\mathcal C}$ which assigns
to every object $X$ the same object $X$ but now viewed as an object of $\overline{\mathcal C}$.
To distinguish it from $X$, we denote this new (or, rather, not so new)
object by $\overline{X}$.

We would now like to understand the structure of $\overline{\mathcal{C}}$ more concretely. For this we will need the following lemma.

\begin{lem}\label{bens} (D. Benson) (i) Let $X = \bigoplus_{i=1}^m X_i$, $Y = \bigoplus_{j=1}^n Y_j$, where $X_i,Y_j\in \mathcal C$ are indecomposable. Let $f: X \rightarrow Y \in \mathcal{C}$ be a morphism, and let $f = \bigoplus_{i,j} f_{ij}$ where $f_{ij}: X_i \rightarrow Y_j$. Then, $f$ is negligible if and only if $f_{ij}$ is negligible for all $i, j$.

(ii) Assume that $X,Y$ are indecomposable. Then $f$ is negligible if and only if either $\dim Y = 0$ or $f$ is not an isomorphism.
\end{lem}

\begin{proof} We first prove (ii).
Suppose $f: X \rightarrow Y$ is not an isomorphism and $g: Y \rightarrow X$. Assume $f \circ g$ is an isomorphism. Replacing $g$ with its composition with the inverse of this isomorphism, we may assume that $f\circ g=1_Y$. Then $g\circ f: X\to X$ is a projector. Thus there is an isomorphism $X\cong Y\oplus T$ such that $f$ is the natural projection. Since $f$ is not an isomorphism, $T\ne 0$, which contradicts indecomposability of $X$. Therefore, $f \circ g$ is not an isomorphism. 

Since $Y$ is indecomposable, ${\rm End}(Y)$ is a local algebra. Hence $f\circ g$ is nilpotent. It follows that ${\rm Tr}(f \circ g) = 0$, thus $f$ is negligible.
\par
If $\dim Y = 0$ then for any $g: Y \rightarrow X$ write $f \circ g = \lambda 1_Y + \eta$, where $\lambda \in \KK$ and $\eta: Y \rightarrow Y$ is nilpotent. Hence,
$$
{\rm Tr}(f \circ g) = {\rm Tr}(\lambda 1_Y + \eta) = \lambda \dim Y = 0,
$$
so $f$ is negligible. This completes one direction.
\par
Now let us establish the other direction. Suppose $f$ is negligible. It suffices to show that if $f$ is an isomorphism then $\dim Y =  0$. If $f$ is an isomorphism, let $g = f^{-1} : Y \rightarrow X$. By the negligiblity assumption, we have $0 = {\rm Tr}(f \circ g)= \dim Y$.
\par
Now we prove (i). If $g : Y \rightarrow X$ decomposes as $g = \bigoplus_{j,i} g_{ji} : Y_j \rightarrow X_i$, then ${\rm Tr}(f \circ g) = \sum_{i,j} {\rm Tr}(f_{ij} \circ g_{ji})$. Now, if for all $i, j$ either $\dim Y_j = 0$ or $f_{ij}$ is not an isomorphism, then (ii) implies that ${\rm Tr}(f_{ij} \circ g_{ji}) = 0$ for all $i,j$ and hence ${\rm Tr}(f \circ g ) = 0$, meaning $f$ is negligible. If, on the other hand, there exist $i,j$ such that $\dim Y_j \neq 0$ and $f_{ij}$ is an isomorphism, then let $g_{ji} = f_{ij}^{-1}$, and $g_{qp} = 0$ for $(p, q) \neq (j, i)$; then ${\rm Tr}(f \circ g) = \dim Y_j \neq 0$, which implies $f$ is not negligible. This proves the lemma.
\end{proof}

\begin{cor}
The semisimplification $\overline{\mathcal{C}}$ is a semisimple symmetric tensor category. Moreover, the simple objects of $\overline{\mathcal{C}}$ are the indecomposables of $\mathcal C$ of nonzero dimension.
\end{cor}
\begin{proof} By Lemma \ref{bens}, 
If $X, Y \in\mathcal{C}$ are indecomposable, then $\Hom_{\overline{\mathcal{C}}}(X, Y) = 0$ if $X \ncong Y$ or $\dim X = 0 $. If, on the other hand, $X \cong Y$ and $\dim X \neq 0$, then $\dim_\KK \Hom_{\overline{\mathcal{C}}}(X, Y) = 1$. In particular, if $X\in \mathcal C$ is indecomposable then $\overline{X}$ is simple if $\dim X\ne 0$
and $\overline{X}=0$ if $\dim X=0$, as claimed.
\end{proof}

Motivated by this, we will say that an object $X\in \mathcal{C}$ is {\bf negligible}
if $\overline{X}=0$. In other words, $X$ is negligible if and only if so is $1_X$. Otherwise formulated, $X$ is negligible if all its indecomposable direct summands have dimension $0$.

\begin{rem} 1. If $\mathcal C$ is semisimple then $\mathcal{N}(\mathcal C)=0$, so
$\overline{\mathcal C}=\mathcal C$.

2. If $\mathcal{C}$ is a non-semisimple symmetric tensor category
then the semisimplification functor $S: \mathcal C\to \overline{\mathcal C}$
is not left or right exact, so it is not a tensor functor.
\end{rem}

\begin{rem} An intriguing open question is whether there exists a {\bf semisimple}
symmetric tensor category of non-moderate growth in positive characteristic. Indeed, the Deligne categories $\Rep^{\rm ab} GL_t$, $t\in \Bbb Z_p$ and similar interpolation categories for $O_n,Sp_n,S_n$ have non-moderate growth but are never semisimple.

An obvious idea is to consider the semisimplification
of one of these categories or some (pseudo)tensor subcategory there. However, it is not clear whether or not such a semisimplification will have moderate growth, and in the few examples
where this has been computed it turns out that it actually does (the growth can be drastically reduced by the semisimplification procedure, as it discards many negligible direct summands, declaring them to be zero).

Another approach to constructing such an example could be taking a sequence of finite groups $G_n$ of order coprime to $p$ and of representations $V_n$ of $G_n$ over $\KK$ such that $\dim V_n\to \infty$ as $n\to \infty$ but the length of $V_n^{\otimes k}$
is bounded above by some constant $\ell_k$ independent of $n$.
In this case the category $\langle V\rangle$ generated by
$V:=(V_1,V_2,...)$ in the ultraproduct of $\Rep_\KK(G_n)$ will be a semisimple
symmetric tensor category of non-moderate growth, as desired. However, this approach fails, as it can be deduced from the theory of finite groups that sequences $(G_n,V_n)$ with such properties do not exist.

This lack of success in constructing a counterexample motivates a conjecture that any semisimple symmetric tensor category in characteristic $p$ has moderate growth.
\end{rem}

\subsection{The Verlinde Category $\Ver_p$}
A simple but rich example of semisimplification is the {\bf Verlinde category} $\Ver_p$. Let $\KK$ be an algebraically closed field of characteristic $p$. Let $\mathcal{C} = \Rep_\KK (\Z/p)$ be the category of finite dimensional $\KK$-representations of the cyclic group $\Z/p$. This is the simplest group whose representations are not semisimple in characteristic $p$. Indeed, since $\Z/p= \langle g \ | \ g^p = 1 \rangle$, we can realize the group algebra $\KK[\Z/p]$ as $\KK[g]/(g^p - 1) = \KK[g]/(g-1)^p$, since we are in characteristic $p$. By the Jordan normal form theorem, the pairwise non-isomorphic indecomposable representations of $\KK[\Z/p\Z]$ are $J_n := \KK[g]/(g-1)^n$ for $1 \leq n \leq p$. The representation $J_n$ can be realized as $\KK^n$ with $\Z/p$ action given by mapping $g $ to the $n \times n$ Jordan block with eigenvalue $1$.

 Define now the {\bf Verlinde category} $\Ver_p$ to be the semisimplification of $\mathcal{C}$. Then the simple objects of ${\rm Ver}_p$ are $L_i:=\overline{J_i}$,
 $i=1,...,p-1$. Indeed, the remaining indecomposable object $L_p$ has dimension $0$ in characteristic $p$ so it vanishes in the semisimplification.

The tensor product $L_m \otimes L_n$ for $m,n\ll p$ decomposes like the tensor product of Jordan blocks in characteristic $0$, which follows the usual {\bf Clebsch-Gordan rule} for representations of $\mathfrak{sl}_2(\mathbb{C})$:
\[
L_m \otimes L_n = \bigoplus_{i=1}^{\min(m,n)} L_{|m-n| + 2i - 1}.
\]
 However, when $m,n$ get close to $p$, this direct sum has to be truncated.
 Obviously there is no $L_i$ for $i \ge p$, but some other terms also disappear,
 following the so called {\bf truncated Clebsch-Gordan rule}, or {\bf Verlinde rule}:
 \[
L_m \otimes L_n =\bigoplus_{i=1}^{\min(m,n,p-m,p-n)} L_{|m-n| + 2i - 1},
\]
which arises in the representation theory for the affine  Lie algebra $\widehat{\mathfrak{sl}}_2$ at level $k=p-2$ and Lusztig's quantum group $U_q(\mathfrak{sl}_2)$ for $q=e^{\frac{\pi i}{p}}$ (see \cite{EGNO}, Subsection 8.18.2). This explains the terminology.

\begin{example} When $p = 2$, there is one simple object, and the category produced is $\Ver_2 = \Vecc_\KK$.

When $p = 3$, there are two simple objects: $L_1$ and $L_2$, where $L_1$ is the unit and $L_2 \otimes L_2 = L_1$. It turns out that the usual braiding on $\mathcal{C} = \Rep_\KK(\Z/p)$ yields the super-braiding from \S\ref{supervec} and so $\Ver_3 = \sVec_\KK$.

For $p = 5$, we have four simple objects: $L_1, L_2, L_3, L_4$ where $L_1$ is the unit. The truncated Clebsch-Gordan rule in $\Ver_p$ tells us that, for instance,
\begin{equation}\label{tenpr}
L_3 \otimes L_3 = \mathbbm{1} \oplus L_3.
\end{equation}
From this relation, we can actually deduce that $\Ver_5$ is not super-Tannakian. Indeed, if $F: \Ver_5 \rightarrow \sVec_\KK$ is a fiber functor and $\dim_\KK F(L_3)=d$ then equality \eqref{tenpr} means that $d^2 = 1 + d$. This equation has no solutions over the integers, which is where dimensions in $\sVec_\KK$ take values. This shows that Deligne's theorem does not hold in characteristic $5$ (as $\Ver_5$ definitely is a symmetric tensor category of moderate growth).

In a similar way one can show that $\Ver_p$ is not super-Tannakian for any $p\ge 5$, hence Deligne's theorem fails for all such $p$.
We will see below that it fails in characteristics 2 and 3 as well, but in this case the examples will have to be non-semisimple.
\end{example}

\begin{rem} 1. We have $L_{p-1}\otimes L_{p-1}=L_1$ and $\dim L_{p-1}=-1$, so for $p>2$ the objects $L_1$ and $L_{p-1}$ span a copy of ${\rm sVec}_\KK$ inside ${\rm Ver}_p$. Let ${\rm Ver}_p^+$ be the tensor subcategory of ${\rm Ver}_p$ spanned by $L_i$ with odd $i$. Then
for $p>2$ we have ${\rm Ver}_p={\rm Ver}_p^+\boxtimes {\rm sVec}_\KK$.

2. The category ${\rm Ver}_p$ also arises as the semisimplification of the category $\Rep\KK[x]/(x^p)$, the representation category of
the finite group scheme $\alpha_p={\rm Spec} \KK[x]/(x^p)$. Thus if we have a
Lie algebra $\g$ over $\KK$ with an automorphism $g$ such that $g^p=1$
or a derivation  $d$ such that $d^p=0$ then the semisimplification $\overline{\g}$
is a Lie algebra in ${\rm Ver}_p$, so in particular
if $\g_i:=\Hom(L_i,\g)$ then $\g_1\oplus \g_{p-1}$ is a Lie superalgebra.
It turns out that this is a source of many interesting examples of Lie superalgebras in positive characteristic, and more generally this construction is a starting point for Lie theory in ${\rm Ver}_p$.

3. ${\rm Ver}_p$ can also be obtained as the semisimplification of the category of tilting modules over $SL_2(\KK)$, which is how it was originally constructed by S. Gelfand and D. Kazhdan and by G. Georgiev and O. Mathieu in early 1990s. They also introduced similar categories ${\rm Ver}_p(G)$ for other simple algebraic groups $G$ if $p\ge h$ where $h$
is the Coxeter number of $G$. Namely, ${\rm Ver}_p(G)$ is the semisimplification
of the category of tilting modules for $G(\KK)$ (see \cite{EO} and references therein for more details).
 \end{rem}

\subsection{Fiber functors into $\Ver_p$}

We have seen that for a symmetric tensor category of moderate growth over a field $\KK$ of characteristic zero, there is always a fiber functor into $\sVec_\KK$ (Deligne's theorem), but for prime characteristic this is not true; e.g., $\Ver_p$ was a counterexample in characteristic $p\ge 5$. In fact, even more is true:
$\Ver_p$ is {\bf incompressible} in the sense that it does not admit a fiber functor into a smaller category. More precisely, any tensor functor $H: {\rm Ver}_p\to \mathcal C$
into another symmetric tensor category is necessarily a fully faithful embedding (\cite{BEO}, Theorem 4.71).

It is therefore natural to ask whether we could generalize Deligne's theorem to positive characteristic  if we take the receptacle of fiber functors to be ${\rm Ver}_p$
instead of $\sVec_\KK$; in other words, if $\mathcal{C}$ is a symmetric tensor category of moderate growth over $\KK$, can we find a fiber functor $\mathcal{C}\to \Ver_p$? If so, this would be a very good news, since then $\mathcal{C}$ would be the representation category of an affine group scheme in $\Ver_p$, and could therefore be studied by methods of Lie theory.

Amazingly, it turns out that this is true in the case $\mathcal{C}$ is a semisimple category. 
Namely, this is guaranteed by the following theorem: 

\begin{thm}(\cite{CEO}, 2021)\label{ostconj} If $\mathcal{C}$ is a semisimple symmetric tensor category of moderate growth over a field $\KK$ of characteristic $p > 0$, then it admits a fiber functor to $\Ver_p$.
\end{thm}

This theorem was conjectured by V. Ostrik in 2015 (\cite{O}, Conjecture 1.3) after he proved it in the special case of {\bf fusion} categories, i.e., semisimple tensor categories with finitely many simple objects (\cite{O}, Theorem 1.5). 

As noted above, we do not know if the assumption of moderate growth is really needed here.

\begin{rem} As in the case of usual fiber functors, if the fiber functor exists, it is unique up to an isomorphism.
\end{rem}

\subsection{The non-semisimple case} 

Theorem \ref{ostconj} does not hold if $\mathcal{C}$ is not semisimple. The simplest counterexample is in characteristic $2$: we can take $\mathcal{C}$
to be the category $\Rep(\KK[d]/d^2,R)$ of representations of the Hopf algebra $\KK[d]/d^2$ with $\Delta(d)=d\otimes 1+1\otimes d$ with the braiding
$c=\sigma R$, where $\sigma$ is the usual swap and
$R=1\otimes 1+d\otimes d$ (this example, due to S. Venkatesh,
is discussed in \cite{BE}). The examples in characteristic $p>2$ are more complicated.
The ultimate result, proved in \cite{BE} in characteristic $2$ and then in \cite{BEO} for general $p$, is as follows.

\begin{thm}(\cite{BEO}) There is a nested sequence of incompressible symmetric tensor categories in characteristic $p$:
$$
\Ver_p \subseteq \Ver_{p^2} \subseteq \Ver_{p^3} \cdots
$$
The category $\Ver_{p^n}$ arises as the reduction to characteristic $p$ of the semisimplified category of tilting modules over the Lusztig quantum group $ U_q(\mathfrak sl_2)$ at the $p^n$-th root of unity. It also arises
as the abelian envelope of the quotient of the category of tilting modules
over $SL_2(\Bbb K)$ by the tensor ideal $I_n$ generated by the $n$-the Steinberg module
$T_{p^n-1}$ (the tilting module of highest weight $p^n-1$).
\end{thm}

We can thus form the category $\Ver_{p^\infty}:=\bigcup_{n\ge 1}\Ver_{p^n}$. One can then make the following conjecture, which would generalize Deligne's theorem to characteristic $p$:

\begin{conj}(\cite{BEO})\label{beoconj} 
If $\mathcal{C}$ is a symmetric tensor category of moderate growth over a field $\KK$ of characteristic $p > 0$, then it admits a unique fiber functor into $\Ver_{p^{\infty}}$.
\end{conj}

For example, this holds for $\Rep(\KK[d]/d^2,R)$, as it is a tensor subcategory in $\Ver_4$.

\subsection{Frobenius exact categories} 

There is, however, a subclass of not necessarily semisimple symmetric tensor categories 
for which Theorem \ref{ostconj} holds. They are called {\bf Frobenius exact}, or {\bf locally semisimple}. Namely, given a symmetric tensor category $\mathcal{C}$ and 
an object $X\in \mathcal{C}$, we may form the object $X^{\otimes p}\in \mathcal{C}\boxtimes \Rep(\Bbb Z/p)$. Therefore, 
applying the semisimplification functor, we obtain an object $\overline{X^{\otimes p}}$ of  $\mathcal{C}\boxtimes \Ver_p$ which we denote by ${\rm Fr}(X)$ and call the {\bf Frobenius twist} of $X$. It is clear that 
$X\mapsto {\rm Fr}(X)$ is a symmetric monoidal functor (as it is a composition of two symmetric monoidal functors). It is also clear that it maps scalar multiplication by $\lambda\in \KK$ to multiplication by $
\lambda^p$. What is less obvious and more surprising is that ${\rm Fr}$ is an additive functor (as the functor $X\mapsto X^{\otimes p}$ is not additive at all). The reason for this is that the ``non-additive part" of $X\mapsto X^{\otimes p}$ gets killed by the semisimplification functor, when we set the Jordan blocks of size $p$ to zero. The functor ${\rm Fr}$ is called the {\bf Frobenius functor}. It was introduced in \cite{O} and studied in detail in \cite{EO2} and \cite{CEO}.\footnote{We note that the papers \cite{Co2} and \cite{CEO} also discuss a few other kinds of Frobenius functors.} For example, for $p=2$, $\Ver_2={\rm Vec}_\KK$, and ${\rm Fr}:\mathcal{C}\to \mathcal{C}$ is defined as follows: ${\rm Fr}(X)$ is the cohomology of $1-\sigma$ acting on $X\otimes X$, where $\sigma$ is the swap (it is clear that $(1-\sigma)^2=1-\sigma^2=0$). 

Thus the only way ${\rm Fr}$ can fall short of being a symmetric (twisted-linear) tensor functor is that it may fail to be exact. This does, in fact, happen: in the category $\Rep(\KK[d]/d^2,R)$ (and more generally, 
$\Ver_{p^n}$ for $n\ge 2$) the functor ${\rm Fr}$ is neither left nor right exact. It therefore makes sense 
to introduce the following definition. 

\begin{Def} $\mathcal{C}$ is called {\bf Frobenius exact} if its Frobenius functor is exact.  
\end{Def}   
 
For example, any semisimple category is automatically Frobenius exact. Since the Frobenius functor commutes with symmetric tensor functors, it follows that if $\mathcal{C}$ admits a symmetric tensor functor 
into a semisimple tensor category (for example, $\Ver_p$) then it is Frobenius exact (in particular, this applies to any super-Tannakian category). The main result of \cite{CEO} generalizing Theorem \ref{ostconj} is the following theorem. 

\begin{thm}\label{mainceo} (\cite{CEO}, Theorem 1.1) A symmetric tensor category $\mathcal{C}$ of moderate growth is Frobenius exact if and only if it admits a (necessarily unique) fiber functor to $\Ver_p$. 
\end{thm} 

For finite tensor categories (i.e., ones with finitely many simple objects and enough projectives) 
this theorem was proved earlier in \cite{EO2}, Theorem 8.1.   

In practice, the definition of Frobenius exactness is not easy to verify. Therefore, the following criterion is quite useful. 

\begin{prop} (\cite{Co2}, Theorem C). The following conditions on a symmetric tensor category 
$\mathcal{C}$ are equivalent. 

(i) $\mathcal{C}$ is Frobenius exact. 

(ii) For each filtered object $X\in \mathcal{C}$, the canonical epimorphism $S^\bullet({\rm gr}X)\to {\rm gr}(S^\bullet X)$ is an isomorphism.

(iii) For each monomorphism $\mathbbm{1}\to X$, the induced morphism $\mathbbm{1}\to S^pX$ is non-zero.

(iv) There exists an abelian $\KK$-linear symmetric monoidal category $\mathcal{D}$ (not necessarily a tensor category!) and an exact $\KK$-linear symmetric monoidal functor $F:\mathcal{C}\to \mathcal{D}$ which splits every short exact sequence in $\mathcal{C}$. 
\end{prop} 

The last property justifies the term ``locally semisimple" introduced by Deligne. 

In particular, all these properties fail in the categories $\Ver_{p^n}$, $n\ge 2$, as they are not Frobenius exact. 

\section{Appendix}

\subsection{Growth of tensor powers}

Theorem \ref{ostconj} may be used to study the growth of lengths of tensor powers in symmetric tensor categories, similarly to the analysis of \cite{BS}.\footnote{Here we partly follow the discussion of \cite{CEO}, Section 4.} Namely, for a symmetric tensor category $\mathcal C$ of moderate growth and $V\in \mathcal C$ let $d_n(V)$ be the length of $V^{\otimes n}$. It is easy to see that $d_{n+m}(V)\ge d_n(V)d_m(V)$, i.e., the sequence $\lbrace d_n(V)^{-1}\rbrace $ is submultiplicative (indeed, because of rigidity the tensor product of two simple objects cannot be zero). Therefore, by Fekete's lemma (\cite{B}, Lemma 1.6.3),
there exists a limit
$$
{\rm gd}(V):=\lim_{n\to \infty}d_n(V)^{\frac{1}{n}},
$$
which we will call the {\bf growth dimension} of $V$. Moreover, if $V\ne 0$ then 
$$
1\le {\rm gd}(V)={\rm sup}_{n\ge 1}d_n(V)^{\frac{1}{n}}<\infty.
$$
Also it is easy to see that 
$$
{\rm gd}(V^*)={\rm gd}(V),\ {\rm gd}(V^{\otimes n})={\rm gd}(V)^n,\ {\rm gd}(V\otimes W)\ge {\rm gd}(V){\rm gd}(W).
$$ 
Finally, it is less trivial but still not hard to show that if 
$$
0\to V\to U\to W\to 0
$$
is a short exact sequence then 
$$
{\rm gd}(U)\ge {\rm gd}(V)+{\rm gd}(W)
$$
(see \cite{CEO}, Lemma 4.9). 
 
\begin{lem}\footnote{This lemma fails for non-symmetric categories. For example, if $X$ is the standard 2-dimensional representation of the Yangian $Y(\mathfrak{sl}_2)$ then $X^{\otimes n}$ 
is simple for all $n$.}\label{square} Let $\mathcal C$ be a symmetric tensor category in any characteristic and $X\in \mathcal C$ 
be such that $X\otimes X$ is simple. Then $X$ is invertible, i.e., $X\otimes X^*=\mathbbm{1}$. 
\end{lem} 

\begin{proof} The symmetric braiding $\sigma_X: X\otimes X\to X\otimes X$ is a scalar such that $\sigma_X^2=1$. 
So $\sigma_X=\sigma_{X^*}=\pm 1$. Hence $\sigma_{X\otimes X^*}=\sigma_X\sigma_{X^*}=1$. 

Let $M:=X\otimes X^*/\mathbbm{1}$, where 
$\mathbbm{1}$ is embedded into $X\otimes X^*$ by the coevaluation map. Then 
$X\otimes X^*$ has a 3-step filtration with successive quotients 
$M\otimes M,M\oplus M,\mathbbm{1}$. The map $\sigma_{X\otimes X^*}$ preserves this filtration and acts on the second quotient $M\oplus M$ by permuting the two copies of $M$. 
So the identity $\sigma_{X\otimes X^*}=1$ implies that $M=0$, i.e., 
$X\otimes X^*\cong \mathbbm{1}$. 
\end{proof} 

\begin{cor}\label{lowerbo} 1. The following conditions on an object $V\in \mathcal C$ are equivalent: 

(i) ${\rm gd}(V)=1$; 

(ii) $d_n(V)=1$ for all $n$;

(iii) $V$ is invertible.

2. If ${\rm gd}(V)>1$ then ${\rm gd}(V)\ge \sqrt{2}$. Moreover, if $p>2$ then ${\rm gd}(V)\ge \frac{1+\sqrt{5}}{2}$.\footnote{Note that in characteristic $2$ the equality ${\rm gd}(V)=\sqrt{2}$ is possible, e.g. 
in the category $\mathcal C_2=\Ver_4$ discussed in \cite{BE}. Also ${\rm gd}(V)=\frac{1+\sqrt{5}}{2}$ occurs in 
${\rm Ver}_5$.}    
\end{cor}

\begin{proof} 1. Since  ${\rm gd}(V)={\rm sup}_{n\ge 1}d_n(V)^{\frac{1}{n}}$, (i) implies (ii). By Lemma 
\ref{square}, (ii) (already for $n=2$) implies (iii). 
Finally, if (iii) holds then $1=d_n(V\otimes V^*)\ge d_n(V)d_n(V^*)$, so (ii) and (i) follow. 

2. By Lemma \ref{square}, $V\otimes V$ has at least two composition factors, so ${\rm gd}(V\otimes V)\ge 2$, hence ${\rm gd}(V)\ge \sqrt{2}$. 

If $p>2$ then we claim that $V\otimes V$ cannot consist of two invertible composition factors. 
Indeed, if it does then one of them must be $S^2V$ 
and the other $\wedge^2V$. So if  $d=\dim(V)$ then their dimensions $d(d+1)/2$ and $d(d-1)/2$ are both $\pm 1$. 
But these equations have no solutions for any choice of signs, a contradiction.

Let $\beta$ be the infimum of ${\rm gd}(Y)$ for non-invertible $Y\in \mathcal C$. If $\beta\ge 2$, there is nothing to prove, so assume that $\beta<2$. If $V\otimes V$ has three or more composition factors, we have ${\rm gd}(V)^2\ge 3>1+\beta$. 
Otherwise $V\otimes V$ has two composition factors one of which, call it $W$, is not invertible. So we again get
${\rm gd}(V)^2\ge 1+{\rm gd}(W)\ge 1+\beta$. Thus $\beta^2\ge 1+\beta$. Hence $\beta\ge \frac{1+\sqrt{5}}{2}$.  
\end{proof} 

Note that if $\mathcal C$ has finitely many simple objects 
then ${\rm gd}(V)={\rm FPdim}(V)$ is the Frobenius-Perron dimension 
of $V$ defined in \cite{EGNO}, Chapter 3 (\cite{CEO}, Lemma 8.3). 
Thus the same holds if $\mathcal C$ is a nested union of such tensor categories. 
But in fact, this holds even more generally. 

\begin{thm}\label{gro} Let $F: \mathcal C \to {\rm Ver}_p$ be a symmetric tensor functor, where for $p=0$ we agree that ${\rm Ver}_p={\rm sVec}_\KK$. 

(i) For every $X\in \mathcal C$ there exist $C,r>0$ such that for any simple composition factor 
$Y$ in $X^{\otimes n}$ we have 
$$
{\rm FPdim}(F(Y))\le C n^r.
$$

(ii) For every $X\in \mathcal C$, ${\rm gd}(X)={\rm FPdim}(F(X))$. In particular, 
${\rm gd}$ defines a homomorphism ${\rm Gr}(\mathcal C)\to \Bbb R$ from the Grothendieck ring of $\mathcal{C}$ to $\Bbb R$, and 
${\rm gd}(X)$ is an algebraic integer.  
\end{thm} 

\begin{proof} 
(i) Consider first the case of characteristic $p>0$. We may assume that $\mathcal C$ is tensor generated by $X$. We have $F(X)=\oplus_{i=1}^{p-1}L_i\otimes V_i$, where $V_i$ are finite dimensional vector spaces. 
Let $G=\underline{\rm Aut}_\otimes(F)\subset \underline{GL}(F(X))$ be the corresponding affine group scheme of finite type 
in ${\rm Ver}_p$. Let $G_0\subset \prod_{i=1}^{p-1} \underline{GL}(V_i)$ be the classical part of $G$, i.e., $\mathcal O(G_0)=\mathcal O(G)/I$, where $I$ is the ideal generated by the non-trivial $\Ver_p$-subobjects of $\mathcal O(G)$. Then $G_0$ is a usual affine group scheme and the algebra $\mathcal O(G/G_0)$ is finite dimensional (see \cite{Ve}, Proposition 7.10); indeed, the algebra $\mathcal O(G/G_0)$ is generated by finitely many $L_i$ for $i>1$ and the symmetric algebra $S(L_i)$ is finite dimensional for such $i$. 

Now let $V$ be a simple $G$-module and $Y$ a simple quotient
of $V|_{G_0}$. Then by Frobenius reciprocity $\Hom_G(V, {\rm Coind}_{G_0}^G Y)=\Hom_{G_0}(V|_{G_0},Y)\ne 0$, so 
$V$ is a subobject of ${\rm Coind}_{G_0}^G Y$. Thus 
$$
\ell(F(V))\le \ell(F({\rm Coind}_{G_0}^G Y))=\ell(F((\mathcal O(G)\otimes Y)^{G_0}))=\ell(F(\mathcal O(G/G_0)))\dim_\KK Y,
$$
as $G_0$ acts freely on $G$ (see \cite{Ve}, Proposition 7.12). But $Y$ is a composition factor in $(X|_{G_0})^{\otimes n}$. So it suffices to establish the required bound for representations of $G_0$. In other words, it is enough to show that for every simple $G_0$-module $X$ there exist $C,r$ such that 
for every composition factor $Y$ of $X^{\otimes n}$ we have 
\begin{equation}\label{inequa1}
\dim_\KK Y\le Cn^r.
\end{equation} 
Thus it is enough to establish \eqref{inequa1} for the tautological representation $X=\KK^m$ of   
$GL_m(\KK)$. But this is easy, since dimensions of Weyl modules for $GL_m(\KK)$, given by the Weyl dimension formula, grow polynomially with the highest weight. 

The same proof applies to the case ${\rm char}(\KK)=0$, where instead of \cite{Ve} we can use the well-known fact that the even part $G_0$ of an algebraic supergroup $G$ is of finite index in $G$. 

(ii) follows immediately from (i).  
\end{proof}

\begin{rem} Another proof of Theorem \ref{gro}(ii) is given in \cite{CEO}, Lemma 8.5. In fact, 
there a stronger result is proved: the target category $\Ver_p$ can be replaced by 
any symmetric tensor category with finitely many simple objects. 
\end{rem} 

Note that by Deligne's theorem, in characteristic zero the functor $F$ for a category of moderate growth always exists, so 
the conclusions of Theorem \ref{gro} always hold. Also, 
Theorem \ref{mainceo} implies the following result in characteristic $p$.

\begin{thm}\label{semiconj} In any Frobenius exact (in particular, semisimple) symmetric tensor category 
of moderate growth the map $V\mapsto {\rm gd}(V)$ defines a homomorphism ${\rm Gr}(\mathcal C)\to \Bbb R$, 
and ${\rm gd}(V)$ is an algebraic integer.  
\end{thm} 

We actually expect that this theorem holds even without the Frobenius exactness assumption. In particular, by \cite{CEO}, Lemma 8.5, it holds for categories satisfying Conjecture \ref{beoconj}.

\subsection{Some applications to modular representation theory}

Now we will discuss some applications of the theory of tensor categories to modular representation theory, partly following \cite{CEO}, Section 8. More precisely, we will consider the problem of describing non-negligible indecomposable summands in $V^{\otimes n}\otimes V^{*\otimes m}$, where $V$ is a finite dimensional representation of a finite group $G$ (or, more generally, affine group scheme) over a field of characteristic $p$. This type of questions is discussed in \cite{B,B2,BS}.

\subsubsection{Tensor powers of a representation}
Let $G$ be an affine group scheme over an algebraically closed field $\KK$ of characteristic $p$. Let $V\in \Rep(G)$. Let $\overline d_n(V)$ be the number of indecomposable non-negligible direct summands in $V^{\otimes n}$ (counted with multiplicities). We have $\overline d_n(V)\le (\dim_\KK V)^n$. It is clear 
that $\overline d_n(V)=d_n(\overline V)$, where $\overline V$ is the image of $V$ in the semisimplification of $\Rep(G)$, so this sequence enjoys the properties described in the previous subsection. In particular, 
we can define 
$$
\delta(V):=\lim_{n\to \infty}\overline d_n(V)^{1/n}={\rm gd}(\overline V).
$$ 

Let $q:=e^{\frac{\pi i}{p}}$, and $[n]_q:=\frac{q^n-q^{-n}}{q-q^{-1}}=\frac{\sin \frac{\pi n}{p}}{\sin \frac{\pi}{p}}$. It is easy to see that if $a,b$ are positive integers with $a+b\le \frac{p-1}{2}$ then 
$$
[a]_q+[b]_q\ge [a+b]_q.
$$ 
Hence if $a_1,...,a_r$ are positive integers 
with $\sum_j a_j=d\le \frac{p-1}{2}$ then 
$$
\sum_j [a_j]_q\ge [d]_q.
$$
Since $[a]_q$ increases with $a$ as $1\le a\le \frac{p-1}{2}$, this implies that 
\begin{equation}\label{inequa} 
1\le a_j,d\le \frac{p-1}{2},\ \sum_j a_j\ge d\implies \sum_j [a_j]_q\ge [d]_q.
\end{equation} 

\begin{thm}\label{finti}\footnote{Parts (i),(ii) and (v) are in \cite{CEO}, Theorem 8.15.} (i) There exist unique non-negative integers $m_j$, $j=1,...,p-1$ such that
$$
\delta(V)=\sum_{k=1}^{p-1}[k]_qm_k
$$
and for $p>2$
$$
\delta(S^2V)-\delta(\wedge^2V)=\sum_{k=1}^{p-1}[k]_{q^2}m_k.
$$

(ii) $\dim_\KK V-\sum_{k=1}^{p-1}km_k$ is divisible by $p$. 

(iii) If $\dim_\KK V\le p-1$ then $\dim_\KK V=\sum_{k=1}^{p-1}km_k$. Thus, either $m_{p-1}=0$, so $\dim_\KK V=\sum_{k=1}^{p-2} km_k,$
or $\dim_\KK V=p-1$ and $m_1=...=m_{p-2}=0$, $m_{p-1}=1$, in which case
$\delta(V)=1$.

(iv) If $G$ is of finite order divisible by $p$ (i.e., $\dim \mathcal O(G)$ is finite and divisible by $p$) and $V$ is faithful then
$\delta(V)<\dim_\KK V$.

(v) If $\dim V=d\in \Bbb F_p$, $d\ne 0$, viewed as a number $1\le d\le p-1$ in $\Bbb Z$, then
$$
\delta(V)\ge [d]_q.
$$
\end{thm}

Note that the bound in (v) is sharp (achieved for $G=\Bbb Z/p$). Moreover, the same proof shows that more generally, in any semisimple symmetric tensor category of moderate growth in characteristic $p$, for an object 
$V$ of dimension $d\ne 0$, we have ${\rm gd}(V)\ge [d]_q$. 

\begin{example} For $p=2,3$ we get that $\delta(V)$ is an integer. For $p=5$ we have $\delta(V)=n_1+\frac{1+\sqrt{5}}{2}n_2$ for non-negative integers $n_1,n_2$.
\end{example}

\begin{proof} (i) Let $\mathcal C=\langle \overline{V}\rangle$ and 
$F: \mathcal C\to {\rm Ver}_p$ be the fiber functor. Then $F(\overline{V})=\oplus_{k=1}^{p-1}m_kL_k$ and ${\rm FPdim}(L_k)=[k]_q$, so
$$
{\rm FPdim}(F(\overline{V}))=\sum_{k=1}^{p-1}[k]_q m_k
$$
and for $p>2$
$$
{\rm FPdim}(F(S^2\overline{V}))-{\rm FPdim}(F(\wedge^2\overline V))=
\sum_{k=1}^{p-1}[k]_{q^2}m_k,
$$
see \cite{EOV}, Proposition 4.5.

Let $p>2$. Since $[k]_q=[p-k]_q$, and $[k]_q$ form a $\Bbb Q$-basis of
$\Bbb Q(q+q^{-1})$ when $1\le k\le \frac{p-1}{2}$,
we see that $\delta(V)$ determines $m_k+m_{p-k}$.
On the other hand, applying the Galois automorphism $g$
such that $g(q^2)=-q$, we get
$$
g({\rm FPdim}(F(S^2\overline{V}))-{\rm FPdim}(F(\wedge^2\overline V)))=
\sum_{k=1}^{p-1}(-1)^{k-1}[k]_{q}m_k,
$$
which determines $m_k-m_{p-k}$. Thus $m_k$ are uniquely determined, as claimed.

(ii) The image of $\dim_\KK V$ in $\Bbb F_p$
is $\dim \overline V=\sum_{k=1}^{p-1}km_k\in \Bbb F_p$.
Thus $\dim_\KK V-\sum_{k=1}^{p-1}km_k$ is divisible by $p$.

(iii) Let $d=\dim_\KK V$. If $d<p-1$ then $\wedge^d V=\KK$, $\wedge^{d+1} V=0$,
so $\wedge^d \overline V=\mathbbm{1}$, $\wedge^{d+1}\overline V=0$. Hence
$m_{p-1}=0$ and $d=\sum_k km_k$, as claimed. On the other hand, if $d=p-1$
then still $\wedge^{p-1}\overline V=\mathbbm{1}$. So if $m_{p-1}\ne 0$ then $m_{p-1}=1$
and $m_j=0$ for $j<p-1$. On the other hand, if $m_{p-1}=0$ then
we must have $p-1=\sum_k km_k$ as before.

(iv) Since $G$ is a closed subgroup scheme of $GL(V)$, which in turn is closed in the space $\End V\oplus \End V^*=V\otimes V^*\oplus V\otimes V^*$, the regular $G$-module $\mathcal O(G)$ with action of $G$ by right translations 
is a quotient of $S(V\otimes V^*\oplus V\otimes V^*)$. 
Since $\mathcal O(G)$ is finite dimensional, this implies 
that it is a quotient of $S^{\le r}(V\otimes V^*\oplus V\otimes V^*)$
for some $r$. So, since $\mathcal O(G)$ is projective, it must be a direct summand
of $S^{\le r}(V\otimes V^*\oplus V\otimes V^*)$. Thus every indecomposable 
projective $G$-module $P$ is a direct summand of $S^m(V\otimes V^*\oplus V\otimes V^*)$ for some $m$. So $P$ is a quotient (hence direct summand) of $V^{\otimes m}\otimes V^{*\otimes m}$. But $V^*=\wedge^{\dim V-1}V\otimes \chi$, where $\chi$ is the top exterior power of $V^*$. Thus $P\otimes \chi^{-m}$ is a quotient (hence direct summand) of $V^{\otimes N}$ where $N:=n\dim V$. Thus $\delta(V)\le (\dim(V)^N-\dim(P))^{\frac{1}{N}}$ since $P\otimes \chi^{-m}$ is negligible (as $p$ divides $|G|$). It follows that $\delta(V)<\dim V$.  

(v)  We may assume that $p>2$ and, by tensoring with the odd line if needed, that $1\le d\le \frac{p-1}{2}$. 
We have $\sum_{k=1}^{p-1} km_k=d\in \Bbb F_p$.  
Thus 
$$
\sum_{k=1}^{\frac{p-1}{2}}k(m_k-m_{p-k})=d+pr\in \Bbb Z
$$
for some $r\in \Bbb Z$. Hence 
$$
\sum_{k=1}^{\frac{p-1}{2}}k(m_k+m_{p-k})\ge \sum_{k=1}^{\frac{p-1}{2}}k|m_k-m_{p-k}|\ge |d+pr|\ge d.
$$
Hence by \eqref{inequa}, 
$$
\delta(V)=\sum_{k=1}^{p-1}[k]_qm_k=\sum_{k=1}^{\frac{p-1}{2}}[k]_q(m_k+m_{p-k})\ge [d]_q,
$$
as claimed.    
\end{proof}

\begin{rem} One can also define the invariant $\gamma(V)=\lim_{n\to \infty}\widetilde d_n(V)^{1/n}$, where $\widetilde d_n(V)$ is the {\it dimension} of the non-negligible part of $V$ (\cite{B},1.4). We have $\widetilde d_n(V\otimes W)\le \widetilde d_n(V)\widetilde d_n(W)$, so 
the limit exists by Fekete's lemma, and for a non-negligible $V$, we have 
$$
1\le \gamma(V)={\rm inf}_{n\ge 0}\widetilde d_n(V)^{1/n}\le \dim_\KK V.
$$ 
We also have $\delta(V)\le \gamma(V)$, 
 $\gamma(V^*)=\gamma(V)$,
$\gamma(V)^{\otimes n}=\gamma(V)^n$, 
$\gamma(V\oplus W)\le \gamma(V)+\gamma(W)$, $\gamma(V\otimes W)\le \gamma(V)\gamma(W)$.

Unfortunately, $\gamma(V)$ is much harder to study that $\delta(V)$, even though they are expected to coincide, which would imply that $\gamma=\delta$ is an algebra homomorphism. 
The equality $\gamma(V)=\delta(V)$ would follow from D. Benson's conjecture that dimensions of indecomposable summands in $V^{\otimes n}$ grow polynomially (or at least slower than exponentially) with $n$. 
For example, D. Benson conjectured that if $\delta(V)=1$ then $\gamma(V)=1$, and moreover the 
dimension $\widetilde d_n(V)$ of the unique non-negligible indecomposable summand of $V^{\otimes n}$ grows polynomially with $n$. This is not known even in the simplest nontrivial examples. 
\end{rem} 

\begin{example} Assume $G$ is finite of order divisible by $p$ and $V$ is faithful and indecomposable.

{\bf 1.} Assume $\dim_\KK V=2$. If $p=3$ then $\delta(V)<2$, so we see that $\delta(V)=1$.
In this case $\overline V$ generates ${\rm sVec}_\KK$.
If $p\ge 5$ then we have $\dim_\KK V=2=m_1+2m_2$, $\delta(V)=m_1+[2]_qm_2$.
So the only option is $m_1=1,m_2=0$, i.e., $\delta(V)=[2]_q$ and $\overline V=L_2$
generates ${\rm Ver}_p$. Example: $V$ is the 2-dimensional indecomposable representation of $\Bbb Z/p$.

{\bf 2.} Assume $\dim_\KK V=3$ and $p\ge 5$. We have
$$
3=m_1+2m_2+3m_3,\
\delta(V)=m_1+[2]_qm_2+[3]_qm_3.
$$
So we have

{\bf Case 1.} $m_1=m_2=0,m_3=1$ and $\delta(V)=[3]_q$. Example:
 $V$ is the $3$-dimensional indecomposable representation of $\Bbb Z/p$.

{\bf Case 2.} $m_1=m_2=1,m_3=0$ and $\delta(V)=1+[2]_q$.

 To distinguish between these cases, we observe that in Case 1, $S^{p-2}V$ is negligible, while in Case 2 it is not. Using this criterion, it can be shown that Case 2 does not actually occur.\footnote{D. Benson, private communication.}

{\bf 3.} Assume that $\dim_\KK V=4$ and $p\ge 5$. Then we are in one of the following cases.

{\bf Case 1.} $m_1=m_2=m_3=0,m_4=1$, $\delta(V)=[4]_q$. So $V$ generates
${\rm sVec}_\KK$ for $p=5$ and ${\rm Ver}_p$ for $p\ge 7$.
Example: $G=\Bbb Z/p$, $V$ the $4$-dimensional indecomposable representation.

{\bf Case 2.} $m_1=0,m_2=2,m_3=0$, so $\delta(V)=2[2]_q$. Example:
$G=\Bbb Z/2\ltimes (\Bbb Z/p)^2$ (action of $\Bbb Z/2$ by swap), and
$V=V_1\oplus V_2$, where $V_1,V_2$ are the 2-dimensional indecomposable representations of the two copies of $\Bbb Z/p$ (on which the other copy acts trivially).

{\bf Case 3.} $m_1=1, m_2=0, m_3=1$, so $\delta(V)=1+[3]_q=[2]_q^2$.
Example: same as Case 2 except $V=V_1\otimes V_2$.

{\bf Case 4.} $m_1=2, m_2=1,m_3=0$, so $\delta(V)=2+[2]_q$. We expect that this never happens.
\end{example}

\subsubsection{Characteristic 2}
Consider now the case of characteristic $p=2$. In this case, according
to Theorem \ref{ostconj}, the category  $\mathcal C=\langle \overline V\rangle$ is Tannakian.
Thus, $\mathcal C=\Rep(\overline G_V)$ for a suitable linearly reductive finite type affine group scheme $\overline G_V$ (i.e., one whose representation category is semisimple). By a theorem of Nagata (\cite{N}), such a group scheme can be included in a short exact sequence
\begin{equation}\label{exsec}
1\to A_V^\vee\to \overline G_V\to \Gamma_V\to 1,
\end{equation}
where $\Gamma_V$ is a finite group of odd order, $A_V$  is a finitely generated abelian group without odd torsion, and $A_V^\vee$ is the dual group scheme of $A_V$.

This is closely related to the following conjecture of D. Benson (\cite{B2}, Conjecture 1.1), supported by ample computer evidence and proofs in special cases.

\begin{conj}\label{beco}
  Let $G$ be a finite $2$-group and $V$ an indecomposable odd dimensional representation of $G$ over an algebraically closed field of characteristic $2$ and let us decompose $V \otimes V^{*}$ as the direct sum $\KK \oplus Q$. Then, $Q$ is a direct sum of even-dimensional indecomposable representations.\footnote{In fact, it is conjectured in \cite{B2} that the dimensions of the summands in $W$ are moreover divisible by $4$.}
\end{conj}

In the language of tensor categories, Conjecture \ref{beco} says that the semisimplification
$\overline{{\rm Rep}_\KK(G)}$ is a {\bf pointed category}, i.e., the symmetric tensor category
${\rm Vec}_{A}$ of vector spaces graded by an abelian group $A=A_G$, and in particular for any indecomposable odd-dimensional $V\in \Rep_\KK(G)$, the subcategory
$\langle \overline V\rangle$ is ${\rm Vec}_C$ for a cyclic group $C$.

Moreover, computer evidence collected by D. Benson (private communication) suggests that $C=\Bbb Z$ or has order a power of $2$. Such strengthened version of the conjecture would simply be equivalent to

\begin{conj}\label{gammatriv} If $G$ is a finite 2-group then we always have $\Gamma_V=1$. In other words,
 $\overline{\Rep_\KK(G)}={\rm Vec}_A$ for some abelian group $A$
without odd torsion.
\end{conj}

For a general finite group $G$ and ${\rm char}(\KK)=2$,
Theorem \ref{ostconj} and Nagata's theorem imply that $\overline{\Rep_\KK(G)}={\rm Vec}_A^\Gamma$ where
$A$ is an abelian group without odd torsion and $\Gamma$ is a projective limit of finite groups of odd order, where 
the superscript denotes $\Gamma$-equivariantization, see \cite{EGNO}, Subsection 4.15.
On the other hand, it is known (\cite{EO}, Section 4) that in any characteristic $p$
$$
\overline{{\rm Rep}_{\KK}(G)}\cong \overline{{\rm Rep}_\KK(N(G_p))}\cong \overline{{\rm Rep}_\KK(G_p)}^{N(G_p)/G_p},
$$
where $G_p$ is the Sylow $p$-subgroup of $G$, $N(G_p)$ its normalizer, and the superscript $N(G_p)/G_p$ means taking the $N(G_p)/G_p$-equivariantization.
So Conjecture \ref{gammatriv} is equivalent to the following seemingly more general conjecture.

\begin{conj}\label{gammatriv1} $\Gamma=N(G_2)/G_2$.
\end{conj}

Indeed, since Conjecture \ref{gammatriv} says that
$\overline{{\rm Rep}_\KK(G_2)}={\rm Vec}_{A}$ for an abelian $A$ without odd torsion, we have
 $\overline{{\rm Rep}_\KK(G)}={\rm Vec}_{A}^{N(G_2)/G_2}$.
So $\Gamma=N(G_2)/G_2$.

A priori (from Theorem \ref{ostconj}), we only have
$\overline{{\rm Rep}_\KK(G_2)}={\rm Vec}_{A}^{\Gamma'}$
for some group $\Gamma'$ which is the projective limit
of odd order groups. Thus we have
a short exact sequence
$$
1\to \Gamma'\to \Gamma\to N(G_2)/G_2\to 1.
$$
So we have a surjective homomorphism $\Gamma\to N(G_2)/G_2$, and Conjecture \ref{gammatriv} says that it is an isomorphism.

Theorem \ref{ostconj} implies a weak version of Benson's conjecture. Namely,
we have the following proposition.

\begin{prop} Let $V$ be an indecomposable odd-dimensional representation of a finite group $G$ over $\KK$ of characteristic $2$, and let $\Sigma(V)$ be the collection of all odd-dimensional indecomposables that occur as direct summands in $V^{\otimes n}\otimes V^{*\otimes m}$ for various $n,m$. There exists a constant $K_V$ such that for every $W\in \Sigma(V)$ one has $\delta(W)\le K_V$, and for any $W_1,...,W_r\in \Sigma(V)$, the tensor product $W_1\otimes...\otimes W_r$ contains $\le K_V^r$ odd-dimensional direct summands.\footnote{Benson's conjecture says that one may take $K_V=|N(G_2)/G_2|$ for all $V$.}
\end{prop}

Actually, one can prove the same in any characteristic $p>0$.
This follows from the fact that any linearly reductive affine group scheme $H$ of finite type
in  ${\rm Ver}_p$ contains a torus $T$ which has finite index in $H$ (i.e., $\mathcal O(H/T)=\mathcal O(H)^T$ is finite dimensional), which follows from Nagata's theorem and the results of \cite{Ve}. 

\subsubsection{Characteristic 3} In characteristic $3$, it is conjectured in \cite{B2}
 that if $V$ is an indecomposable representation of a finite 3-group $G$ of dimension coprime to $3$ then $V\otimes V$ contains a direct summand $W$ such that $W\otimes W^*=\mathbbm{1}\oplus Q$, where $Q$ is negligible (i.e., $\overline{W}$ is invertible).
In other words, there exists an indecomposable $W$ such that $V^*$ is the unique
direct summand in $V\otimes W$ of dimension not divisible by $3$.

Let us see how this can be interpreted in the language of tensor categories.

As before, let $\langle \overline V\rangle$ be the category generated by $\overline{V}$. Theorem \ref{ostconj} implies that
$\langle \overline V\rangle=\Rep(\overline G_V,z)$, where $\overline G_V$ is a linearly reductive
affine supergroup scheme. A super-analog of Nagata's  theorem (see \cite{Ma}, Theorem 8.4) implies that $\overline{G}_V$ is actually an ordinary (even) linearly reductive group scheme (although $z$ may not equal to $1$), i.e. it can be included in an exact sequence \eqref{exsec} where $|\Gamma|$ is coprime to $3$ and $A$ has no $p$-torsion for $p\ne 3$. Thus Benson's conjecture can be interpreted as follows:

\begin{conj} Every irreducible representation of $\overline{G}_V$ is self-dual up to tensoring with a character.
\end{conj}

\subsection{Dimensions in ribbon categories in positive characteristic}

In this section we generalize Lemma \ref{infp} to the case of ribbon categories with a quasi-unipotent twist. For basics on ribbon categories we refer to \cite{EGNO}, Chapter 8. 

\begin{thm} Let $\mathcal C$ be a ribbon pseudotensor category over a field of characteristic $p$ in which traces of nilpotent endomorphisms are zero, with twist $v\in {\rm Aut}({\rm Id}_{\mathcal C})$ such that $v^\ell$ is unipotent for some $\ell$ coprime to $p$. 
Let $n$ be the smallest integer such that $p^n-1$ is divisible by $\ell^2$. Then 
for any object $V\in \mathcal C$ we have $\dim V\in \Bbb F_{p^n}$. 
\end{thm} 

\begin{proof} Let $d=\dim V$ and $c: V^{\otimes p^m}\to V^{\otimes p^m}$ be the 
cyclic braid $c=b_1....b_{p^m-1}$ in the braid group $B_{p^m}$ (cf. Subsection \ref{brai}). Then $c^{\ell p^m}$ is unipotent, as it is expressed via the action of $v^\ell$ on $V$ and $V^{\otimes p^m}$. Thus $c^{\ell p^m}-1=(c^\ell-1)^{p^m}$ is nilpotent. Hence $c^\ell-1$ is nilpotent. 
Thus the trace of $c^\ell-1$ on $V^{\otimes p^m}$ vanishes. But we have 
${\rm Tr}_{V^{\otimes p^m}}(1)=d^{p^m}$ while ${\rm Tr}_{V^{\otimes p^m}}(c^\ell)$ is the Reshetikhin-Turaev invariant 
$RT(T_{\ell,p^m})$ of the torus knot $T_{\ell,p^m}$. This invariant can also be written as 
$RT(T_{p^m,\ell})={\rm Tr}_{V^{\otimes \ell}}(C^{p^m})$, where $C: V^{\otimes \ell}\to V^{\otimes \ell}$ 
is the cyclic braid $C=b_1...b_{\ell-1}\in B_\ell$. 

Now, we claim that ${\rm Tr}_{V^{\otimes \ell}}(C^N)={\rm Tr}_{V^{\otimes \ell}}(C^{N+\ell^2})$. Indeed, 
$C^{\ell^2}$ expresses via the action of $v^\ell$ on $V$ and $V^{\otimes \ell}$, so 
it is unipotent and $C^{\ell^2}-1$ is nilpotent. Thus $(C^{\ell^2}-1)C^N$ is nilpotent, hence its trace vanishes. This implies  the statement. 

Finally, we have $p^n-1=r\ell^2$. Thus $RT(T_{p^n,\ell})=RT(T_{1,\ell})=RT({\rm unknot})=d$. Thus $d^{p^n}=d$ as claimed. 
\end{proof}

\end{document}